\documentclass[11pt]{amsart}
\usepackage{a4wide}
\usepackage{amsmath}
\usepackage{amsfonts}
\usepackage{amssymb, mathrsfs}
\usepackage{color}
\usepackage{graphicx}
\newcommand{\beq}{\begin{equation}}
\newcommand{\eeq}{\end{equation}}

\newcommand{\cc}{\mathbf C}
\newcommand{\res}
{\mathop{\hbox{\vrule height 7pt width .5pt depth 0pt
\vrule height .5pt width 5pt depth 0pt}}\nolimits}

\newcommand{\wid}{\mathrm{width}}

\def\R{\mathbb R}
\def\N{\mathbb N}

\def\cal{\mathcal}

\def\H{{\cal H}}

\def\de{\delta}
\def\e{\varepsilon}

\def\Chi#1{\hbox{{\large $\chi$}{\Large $_{_{#1}}$}}}
\newcommand{\medint}{-\kern -,375cm\int}
\newcommand{\medintinrigo}{-\kern -,315cm\int}

\renewcommand{\S}{\mathbb{S}}

\def\Om{\Omega}

\newcommand{\wto}{\rightharpoonup}
\def\pa{\partial}

\def\Div{{\rm div}}

\newcommand{\dist}{{\rm dist}}
 
\providecommand{\U}[1]{\protect\rule{.1in}{.1in}}
\newcommand{\wtos}{\stackrel{*}{\rightharpoonup}}
\numberwithin{equation}{section}
\newtheorem{theorem}{Theorem}

\newtheorem{corollary}[theorem]{Corollary}

\newtheorem{definition}[theorem]{Definition}

\newtheorem{lemma}[theorem]{Lemma}

\newtheorem{proposition}[theorem]{Proposition}
\newtheorem{remark}[theorem]{Remark}

\numberwithin{theorem}{section}

\begin{document}

\title[]{Total positive curvature and the equality case in the relative isoperimetric inequality outside convex domains}
\author{N. Fusco, M. Morini}
\address[N.\ Fusco]{Dipartimento di Matematica e Applicazioni "R. Caccioppoli",
Universit\`{a} degli Studi di Napoli "Federico II" , Napoli, Italy}
\email[Nicola Fusco]{n.fusco@unina.it}
\address[M.\ Morini]{Dipartimento di Scienze Matematiche Fisiche e Informatiche, Universit\`{a} degli Studi di Parma, Parma, Italy}
\email[Massimiliano Morini]{massimiliano.morini@unipr.it}
\maketitle

%
%
%
%

\begin{abstract} We  settle the case of equality for the relative isoperimetric inequality outside any arbitrary convex set with not empty interior.
\end{abstract}

\section{Introduction}
In \cite{cgr} Choe, Ghomi and Ritor\'e proved the following relative isoperimetric inequality outside convex sets, see also \cite{LWW} for an alternative proof and \cite{K} for a generalization to higher codimension.
\begin{theorem}[\cite{cgr}]\label{th:isoperim}
Let $\cc\subset\R^N$ be a closed convex set with nonempty interior. For any set of finite perimeter $\Om\subset\R^N\setminus\cc$ we have
\beq\label{isoperim1}
P(\Om;\R^N\setminus\cc)\geq N\Big(\frac{\omega_N}{2}\Big)^{\frac{1}{N}}|\Om|^{\frac{N-1}{N}}\,.
\eeq
Moreover, if $\cc$ has a $C^2$ boundary and $\Om$ is a bounded set for which the equality in \eqref{isoperim1}  holds, then $\Om$ is a half ball.
\end{theorem}
Here and in what follows $P(\Om;\R^N\setminus\cc)$ denotes the perimeter of a set $\Om$ in $\R^N\setminus\cc$ in the sense of De Giorgi.
As observed by the authors in \cite{cgr} the equality case for general, possibly nonsmooth, convex sets does not follow from their methods as it cannot be handled by a simple approximation argument. However there are many situations in which nonsmooth convex sets naturally appear. For instance, in models of vapor-liquid-solid-grown nanowires the nanotube is often described as a semi-infinite convex cylinder with sharp edges and possibly  nonsmooth cross sections. In these models super-saturated liquid droplets correspond to isoperimetric regions for the relative perimeter outside the cylinder or more in general  for the capillarity energy, see \cite{KCJANC, OHM}. Experimentally it is observed that in some regimes preferred configurations are given  by spherical caps lying on the  top facet of the cylinder. Understanding these phenomena from a mathematical point of view was our first motivation to study the equality cases in \eqref{isoperim1} also for nonsmooth convex obstacles, beside the intrinsic geometric interest of the problem.

 The main result of this paper reads as follows.
\begin{theorem}[The equality case]\label{th:ugual}
Let $\cc\subset\R^N$ be a closed convex set with nonempty interior and let $\Om\subset\R^N\setminus\cc$ be a set of finite perimeter  such that equality holds in \eqref{isoperim1}. Then $\Om$ is a half ball supported on a facet of $\cc$.
\end{theorem}
Observe that, compared to the last part of Theorem~\ref{th:isoperim}, here we don't have any restriction on the convex set $\cc$ and we allow for possibly unbounded competitors. As in \cite{cgr} the starting point in order the get the characterization of the equality case in \eqref{isoperim1} is an estimate of the positive total curvature $\mathcal K^+(\Sigma)$ of a hypersurface $\Sigma\subset\overline{\R^N\setminus\cc}$ when the contact angle between $\pa\cc$ and $\Sigma$ is larger than or equal to a fixed $\theta\in(0,\pi)$.
Here $\mathcal K^+(\Sigma)$ denotes, roughly speaking, the measure of the image of the Gauss map restricted to those points where there exists a support hyperplane, see  Definition~\ref{kappapiunuovo} below. To state more precisely our result we need to introduce some notation: Given $\theta\in(0,\pi)$ we denote by $S_\theta$ the spherical cap
$$
S_\theta:=\{y\in \S^{N-1}:\, y\cdot e_N\geq \cos\theta\}\,.
$$
Moreover, given  $\Sigma\subset\overline{\R^N\setminus\cc}$ and a point $x\in\Sigma$ we denote by $N_x\Sigma$ the {\em normal cone}
$$
N_x\Sigma=\{\nu\in \S^{N-1}:\, (y-x)\cdot \nu\leq 0\text{ for all }y\in \Sigma\}\,,
$$
that is the set of (exterior) normals to support hyperplanes to $\Sigma$. We can now recall the definition of  total positive curvature.

\begin{definition}\label{kappapiunuovo}
Let $\cc$ be a closed convex set with not empty interior, $\Om\subset\R^N\setminus \cc$ a bounded open set  and $\Sigma:=\overline{\pa\Om\setminus \cc}$. The {\em total positive curvature} of $\Sigma$ is given by
$$
\mathcal K^+(\Sigma):=\H^{N-1}\Big(\bigcup_{x\in\Sigma\setminus\cc}N_x\Sigma\Big)\,.
$$
\end{definition}
The aforementioned estimate on the total positive curvature is provided by the following theorem, which will be proved in Section~\ref{sec:tot}. 
\begin{theorem}\label{th:cgr-main1}
Let $\cc\subset\R^N$ be a closed convex set of class $C^1$, $\Om\subset\R^N\setminus \cc$ a bounded open set  and $\Sigma:=\overline{\pa\Om\setminus \cc}$. Let $\theta_0\in (0, \pi)$ such that
\beq\label{cgr-main10}
\nu\cdot\nu_{\cc}(x)\leq\cos\theta_0 \quad \text{whenever $x\in\Sigma\cap\cc,\,\,\nu\in N_x\Sigma$,}
\eeq
where $\nu_{\cc}(x)$ stands for the outer unit normal to $\cc$ at $x$.
 Then,
\beq\label{cgr-main11}
\mathcal K^+(\Sigma)\geq \H^{N-1}(S_{\theta_0})\,.
\eeq
Moreover, let $r>0$ be such that $\Sigma\cap\cc\subset B_r(0)$.  For any $\e>0$ there exists $\delta$, depending on $\e,\theta_0$ and $r$, but not on $\cc$ or $\Om$, such that if
\beq\label{cgr-main11.1}
\nu\cdot\nu_{\cc}(x)\leq\cos\theta_0+\delta \quad \text{whenever $x\in\Sigma\cap\cc,\,\,\nu\in N_x\Sigma$,}
\eeq
and
\beq\label{cgr-main11.5}
\mathcal K^+(\Sigma)\leq \H^{N-1}(S_{\theta_0})+\delta\,,
\eeq
then  $\Sigma\cap\cc$ is not empty, $\wid(\Sigma\cap\cc)\leq\e$ and more precisely $\Sigma\cap\cc$ lies between two parallel $\e$-distant hyperplanes orthogonal to  $\nu_\cc(x)$  for some $x\in\Sigma\cap\cc$. In particular,
 if { \eqref{cgr-main10}
is satisfied and} the equality in \eqref{cgr-main11} holds, then $\Sigma\cap\cc$ is not empty and lies on a support hyperplane to $\cc$.  
\end{theorem}
Note that in the previous statement  $\wid(\Sigma\cap\cc)$ denotes the distance between the closest pair of parallel hyperplanes which contains $\Sigma\cap\cc$ in between them, see \eqref{eqwid}. Even though the proof of this theorem follows the general strategy of \cite{cgr0} we are able to  improve their result in three directions: (1) we consider a general contact angle $\theta_0\in(0,\pi)$, whereas in \cite{cgr0} only the case $\theta_0=\pi/2$ is considered; (2) we do not assume any regularity on $\Sigma$ and the contact angle condition can  be replaced by  the weaker condition \eqref{cgr-main10}; (3) we get a stability estimate on the `contact part' $\Sigma\cap\cc$ which is \emph{independent of the shape of the convex set }$\cc$. As we will explain below (2) and (3) are crucial in the proof of Theorem~\ref{th:ugual}.

As a consequence of independent interest of the previous theorem we  prove a sharp  inequality for the Willmore energy, see Theorem~\ref{41}.

Before outlining our strategy of the proof of Theorem~\ref{th:ugual} we briefly recall how in \cite{cgr} it is proven that a bounded set $\Om_0$ satisfying the equality  in \eqref{isoperim1} is a half ball, when $\cc$ is sufficiently smooth. There the idea is to consider the isoperimetric profile 
$$
I(m)=\inf\{P(E;\R^N\setminus\cc):\, E\subset\R^N\setminus\cc,\, |E|=m\}\,,
$$
defined for all $m\in(0,|\Om_0|]$, and to show that $I(m)=N\big(\frac{\omega_N}{2}\big)^{\frac{1}{N}}m^{\frac{N-1}{N}}$, that is $I(m)$ coincides with the isoperimetric profile $I_{\mathscr H}(m)$ of the half space.   Moreover, since  $I'(|\Om_0|)=H_\Sigma$, where $H_\Sigma$ is the mean curvature of  $\Sigma=\overline{\pa\Om_0\setminus\cc}$,  
\beq\label{intro1}
\begin{split}
I(|\Om_0|)\big(I'(|\Om_0|)\big)^{N-1}&=\int_{\Sigma\setminus\cc}H_\Sigma^{N-1}\,d\H^{N-1}\geq
(N-1)^{N-1}{\mathcal K}^+(\Sigma)\\
&\geq(N-1)^{N-1}\H^{N-1}(S_{\pi/2})=I_{\mathscr H}(|\Om_0|)\big(I_{\mathscr H}'(|\Om_0|))^{N-1}\,,
\end{split}
\eeq
where the first inequality follows from an application of coarea formula and the geometric-arithmetic mean inequality, see for instance the proof of Theorem~\ref{41}, and the second one follows from the estimate of the total curvature proved in \cite[Lemma 3.1]{cgr}. Now, since $I(m)=I_{\mathscr H}(m)$ for all $m\in[0,|\Om_0|]$, all the inequalities in \eqref{intro1} are equalities. In particular this implies that ${\mathcal K}^+(\Sigma)=\H^{N-1}(S_{\pi/2})$ and that $\Sigma$ is umbilical. From this information, it is not difficult to see that $\Sigma$ must be a half ball.

Note that in the proof of \cite[Lemma 3.1]{cgr}
 it is crucial that the regular part of $\Sigma$ meets $\pa\cc$ orthogonally and in a $C^2$ fashion. This can be inferred from the boundary regularity theory for perimeter minimizers which can be applied only if $\cc$ is sufficiently smooth. Therefore the above argument fails for a general convex set.
 
In order to deal with this lack of regularity we implement a delicate argument based on the approximation of $\cc$  with more regular convex sets.  


Let us describe the argument more in detail. Denote by $\Om_0$ a set of finite perimeter satisfying the equality  in \eqref{isoperim1}. For $\eta>0$ sufficiently small we approximate $\cc$ with the closed $\eta$-neighborhood $\cc_\eta=\cc+\overline{B_\eta(0)}$, which is of class $C^{1,1}$.  Now the idea is to consider the relative isoperimetric
problem in $\R^N\setminus\cc_\eta$. In order to force the minimizers to converge to $\Om_0$ when $\eta\to0$ and the prescribed mass $m$ converges to $|\Om_0|$, we introduce the following constrained isoperimetric profiles with obstacle $\Om_0$:
\beq\label{intro0}
I_\eta(m)=\min\{P(E;\R^N\setminus\cc_\eta):\, E\subset\Om_0\setminus\cc_\eta,\, |E|=m\}
\eeq
for all $m\in(0,|\Om_0\setminus\cc_\eta|]$. Denote by $\Om_{\eta,m}$ a minimizer of the above problem and set $\Sigma_{\eta,m}:=\overline{\pa\Om_{\eta,m}\setminus\cc_\eta}$. Note that  in the general $N$-dimensional case, both the obstacle $\Om_0$ and the minimizers $\Om_{\eta,m}$ may have singularities. Thus, despite the fact that  $\pa\cc_\eta$ is of class $C^{1,1}$, we cannot apply the known boundary regularity results at the points  $x\in\pa\Om_{\eta,m}\cap\pa\cc_\eta\cap\pa\Om_0$.

 However, one useful observation is that $\Om_{\eta,m}$ is  a {\em restricted $\Lambda$-minimizer}, i.e., a $\Lambda$-minimizer with respect to perturbations that do not increase the ``wet part'' $\pa\Om_{\eta, m}\cap \cc_\eta$ (see Definition~\ref{def:lambdamin} below), with a $\Lambda>0$ which can be made  uniform with respect to $\eta$ and locally uniform with respect to $m$ (see Steps 1 and 2 of the proof of Theorem~\ref{th:ugual}).  Another important observation is that  restricted $\Lambda$-minimizers satisfy uniform volume density estimates up to the boundary $\pa \cc_\eta$. All these facts are combined to show that   the constrained isoperimetric profiles \eqref{intro0} are Lipschitz continuous and that 
their derivatives coincide a.e. with the constant  mean curvature $H_{\Sigma^*_{\eta, m}}$ of the regular part $\Sigma^*_{\eta, m}$ of  
$\Sigma_{\eta, m}\setminus \pa\Om_0$ (see Steps 3 and 4). 

As in the argument of \cite{cgr} another important ingredient is represented by the inequality
\beq\label{intro0.5}
\mathcal K^+(\Sigma_{\eta,m})\geq \H^{N-1}(S_{\pi/2})=\frac12 N\omega_N\,,
\eeq
which would  hold  by \cite[Lemma 3.1]{cgr} if we could show that $\Sigma_{\eta,m}$  meets $\pa\cc_\eta$ orthogonally  and in a sufficiently smooth fashion. However, as already observed, due the possible presence of boundary singularities at $\pa\Om_{\eta,m}\cap\pa\cc_\eta\cap\pa\Om_0$ we cannot show that the aforementioned orthogonality condition is attained in a classical sense. An important  step of our argument, which allows us  to overcome this difficulty, consists in showing that restricted $\Lambda$-minimizers satisfy the $\pi/2$  contact angle condition with respect to $\pa \cc_\eta$ in a  ``viscosity'' sense, namely that  the following weak Young's law holds:
 \beq\label{intro1.5}
\nu\cdot\nu_{\cc_\eta}(x)\leq0 \quad \text{whenever $x\in\Sigma_{\eta,m}\cap\cc_\eta,\,\,\nu\in N_x\Sigma_{\eta,m}$\,.}
\eeq
This is achieved in Step 5 by combining a blow-up argument with a variant of the Strong Maximum Principle that we adapted from \cite{dephilippis-maggi-arma}. In turn, owing to \eqref{intro1.5} we may apply Theorem~\ref{th:cgr-main1} to obtain \eqref{intro0.5}. Having established  the latter and with some extra work 
  we can show that $I_\eta(m)\to I_{\mathscr H}(m)$ as $\eta\to0$ for every $m\in(0,|\Om_0|)$, where we recall $I_{\mathscr H}(m)=N\big(\frac{\omega_N}{2}\big)^{\frac{1}{N}}m^{\frac{N-1}{N}}$ is the isoperimetric profile of the halph space (see Steps 6 and 7). 
  
With the convergence of the isoperimetric profiles $I_\eta$ at hand  and using again \eqref{intro0.5}, we can then prove that for a.e. $m\in(0,|\Om_0|)$ 
\beq\label{intro2}
\mathcal K^+(\Sigma_{\eta,m})\to\frac12 N\omega_N\,,
\eeq
and thus $\Sigma_{\eta,m}$ almost satisfies the case of equality in \eqref{cgr-main11}  for $\eta$ sufficiently small. Thanks to the last part of Theorem~\ref{th:cgr-main1} we may then infer that $\Sigma_{\eta,m}\cap\cc_\eta$ is almost flat and with some extra work that the whole wet part $\pa\Om_{\eta, m}\cap \cc_\eta$ has the same property. By showing that for  suitable sequences $m_n\nearrow |\Om_0|$ and $\eta_n\searrow 0$, $\pa\Om_{\eta, m_n}\cap \cc_\eta\to \pa\Om_0\cap\cc$ in the Hausdorff sense, we may finally conclude that   $\pa\Om_0\cap\cc$ is flat and lies on a facet of $\cc$ (see Step 8). 
We highlight here that in all the above argument  it is crucial that the stability estimate on the width of  $\Sigma_{\eta,m}\cap\cc_\eta$  provided by our version Theorem~\ref{th:cgr-main1}  is  independent of the shape of the convex set $\cc_\eta$.  

Having established that the wet part $\pa\Om_0\cap\cc$ is flat, more work is still needed in the final step of the proof to deduce again from \eqref{intro2} that $\Om_0$ is umbilical and in turn a half ball supported on a facet of $\cc$.

The paper is organized as follows: in Section~\ref{sc:2} we collect a few known results of the regularity theory of perimeter quasi minimizers needed in the paper. In Section~\ref{sec:tot} we prove  Theorem~\ref{th:cgr-main1}, while the proof of Theorem~\ref{th:ugual} occupies the whole Section~\ref{sc:4}  with some of the most technical steps outsourced to Section ~\ref{sc:6}. Section~\ref{sc:5} contains further regularity properties if restricted $\Lambda$-minimizers that are needed in the proof of the main result and the proof of the version of the Strong Maximum Principle needed here.
\section{Preliminaries}\label{sc:2}
Throughout the paper
we denote by $B_r(x)$ the ball in $\R^N$  of center $x$ and radius $r>0$. 
In the following we shall often deal with sets of finite perimeter. For the definition and the basic properties of sets of (locally) finite perimeter we refer to the books \cite{AFP, maggi-book}. Here we fix some notation for later use. 
Given $E\subset \R^N$ of locally finite perimeter and a Borel set $G$ we denote by $P(E;G)$ the perimeter of $E$ in $G$. The {\em reduced boundary} of $E$ will be denoted by $\pa^*E$, while $\pa^eE$ will stand for the {\em essential boundary} defined as
$$
\pa^eE:=\R^N\setminus(E^{(0)}\cup E^{(1)})\,,
$$
where $E^{(0)}$ and $E^{(1)}$ are the sets of points where the density of $E$ is $0$ and $1$, respectively. Moreover, we denote by $\nu_E$ the {\em generalized exterior normal} to $E$, which is well defined at each point of $\pa^*E$, {and by $\mu_E$ the {\em Gauss-Green measure} associated to $E$}
\beq\label{gg}
\mu_E:=\nu_E\,\H^{N-1}\res\pa^*E\,.
\eeq
In the following, when dealing with a set of locally finite perimeter $E$, we shall always tacitly assume that $E$ coincides with a precise representative that satisfies the property $\pa E=\overline{\pa^*E}$, see \cite[Remark 16.11]{maggi-book}. A possible choice is given by $E^{(1)}$ for which one may easily check that
\beq\label{Euno}
\pa E^{(1)}=\overline{\pa^*E}\,.
\eeq

We recall the well known notion of perimeter {\em $(\Lambda,r_0)$-minimizer} and the main properties which will be used here.

\begin{definition}\label{def:lambdaminclassic}
Let $\Om\subset\R^N$ be an open set. We say that a set of locally finite perimeter $E\subset\R^N$ is a {\em perimeter $(\Lambda,r_0)$-minimizer} in $\Om$, $\Lambda\geq0$ and $r_0>0$,  if for any ball $B_r(x_0)\subset\Om$, with $0<r\leq r_0$ and any $F\subset\R^N$ such that $E\Delta F\subset\!\subset B_r(x_0) $ we have
$$
P(E;B_r(x_0))\leq P(F;B_r(x_0))+\Lambda|E\Delta F|\,.
$$
\end{definition}
In order to state a useful compactness theorem for $\Lambda$-minimizers we recall that a sequence $\{\mathcal C_n\}$ of closed sets converge in the {\em Kuratoswki sense} to a closed set $\mathcal C$ if the following conditions are satisfied:
\begin{itemize}
\item[(i)] if $x_n\in\mathcal C_n$ for every $n$, then any limit point of $\{x_n\}$ belongs to $\mathcal C$;
\item[(ii)] any $x\in\mathcal C$ is the limit of a sequence $\{x_n\}$ with $x_n\in\mathcal C_n$.
\end{itemize}
One can easily see that $\mathcal C_n\to\mathcal C$ in the sense of Kuratowski if and only if dist$(\cdot,\mathcal C_n)\to$ dist$(\cdot,\mathcal C)$ locally uniformly in $\R^N$. In particular, by the Arzel\`a-Ascoli Theorem any sequence of closed sets admits a subsequence which converge in the sense of Kuratowski.

Throughout the paper, with a common abuse of notation, we write $E_h\to E$ in $L^1$ ($L^1_{loc}$) instead of  $\Chi{E_h}\to\Chi{E}$ in $L^1$ ($L^1_{loc}$). {Moreover, given a sequence of Radon measures $\mu_h$ in an open set $\Om$, we say that $\mu_h\stackrel{*}{\wto}\mu$ {\em weakly* in $\Om$ in the sense of measures}
 if
$$
\int_\Om\varphi\,d\mu_h\to\int_\Om\varphi\,d\mu\qquad\text{for all $\varphi\in C^0_c(\Om)$\,.}
$$
}
Next theorem is a well known result, see for instance \cite[Ch. 21]{maggi-book}. 
\begin{theorem}\label{th:compact}
Let $\Om\subset\R^N$ be an open set and $\{E_n\}$ a sequence of locally finite perimeter sets contained in $\Om$ satisfying the following property: there exists $r_0>0$ such that for every $n$, $E_n$ is a perimeter $(\Lambda_n,r_0)$-minimizer in $\Om$, with $\Lambda_n\to\Lambda\in[0,+\infty)$. Then there exist $E\subset\Om$ of locally finite perimeter and a subsequence $\{n_k\}$ such that
\begin{itemize}
\item[(i)] $E$ is a $(\Lambda,r_0)$-minimizer in $\Om$;
\item[(ii)]  $E_{n_k}\to E$ in $L^1_{loc}(\Om)$,
\item[(iii)] $\pa E_{n_k}\to \mathcal C$ in the Kuratowski sense for some closed set $\mathcal C$ such that $\mathcal C\cap\Om=\pa E\cap\Om$;
\item[(iv)] $\H^{N-1}\res(\pa E_{n_k}\cap\Om)\stackrel{*}{\wto}\H^{N-1}\res(\pa E\cap\Om)$ weakly* in $\Om$ in the sense of measures.
\end{itemize}
\end{theorem}

\begin{remark}\label{rm:compact}
From the definition of Kuratowski convergence it is not difficult to see that (ii) and (iii) of Theorem~\ref{th:compact} imply that, up to extracting a further subsequence if needed, $\overline{E_{n_k}}\to K$ in the sense of Kuratowski, with $K\cap\Om=\overline E\cap\Om$.
\end{remark}

\begin{definition}
Given a set of locally finite perimeter $E$, we say that a function $h\in L^1_{loc}(\pa^*E)$ is the {\em weak mean curvature} of $E$ if for any  vector field $X\in C^1_c(\R^N;\R^N)$ we have
$$
\int_{\pa^*E}\Div_\tau X\,d\H^{N-1}=\int_{\pa^*E}h\, X\cdot\nu_E,d\H^{N-1}\,,
$$
where $\Div_\tau X:=\Div X-(\pa_{\nu_E}X)\cdot\nu_E$ stands for the tangential divergence of $X$ along $\pa^*E$. If such an $h$ exists we will denote it by $H_{\pa E}$.
\end{definition}
Note that if $\pa E$ is of class $C^2$ then $H_{\pa E}$ coincides with the classical mean curvature, or more precisely with the sum of all principal curvatures. In particular, if $E$ coincides locally with the subgraph of a function $u$ of class $C^2$ then locally
$$
H_{\pa E}=-\Div\bigg(\frac{\nabla u}{\sqrt{1+|\nabla u|^2}}\bigg)\,.
$$
Concerning the above mean curvature operator, we recall the following useful Strong Maximum Principle, see for instance \cite[Th. 2.3]{PS}, which covers a more general class of quasilinear equations.

\begin{theorem}\label{th:SMP}
Let $\Om\subset\R^{N-1}$ be an open set and let $u,v\in C^2(\Om)$ such that $u\leq v$ and
$$
\Div\bigg(\frac{\nabla u}{\sqrt{1+|\nabla u|^2}}\bigg)=\lambda=\Div\bigg(\frac{\nabla v}{\sqrt{1+|\nabla v|^2}}\bigg)\,
$$
for some constant $\lambda\in\R$. If $u(x_0)=v(x_0)$ for some $x_0\in\Om$, then $u\equiv v$.
\end{theorem}

We recall the following classical regularity result for $\Lambda$-minimizers. 

\begin{theorem}\label{rm:regola}
Let $E$ be a perimeter $(\Lambda,r_0)$-minimizer  in some open set $\Om\subset\R^N$. Then
\begin{itemize}
\item[(i)] $\pa^*E\cap\Om$ is a hypersurface of class $C^{1,\alpha}$ for every $\alpha\in(0,1)$, relatively open in $\pa E\cap\Om$. Moreover, $\text{dim}_{\H}((\pa E\setminus\pa^* E)\cap\Om)\leq N-8$, where $\text{dim}_{\H}$ stands for the Hausdorff dimension;
\item[(ii)] $H_{\pa E}\in L^\infty(\pa^*E\cap\Om)$, with $\|H_{\pa E}\|_{L^\infty}\leq\Lambda$,  and thus $\pa^*E\cap\Om$ is of class $W^{2,p}$ for all $p\geq1$;
\item[(iii)]  if there exists a $C^1$ hypersurface $\Sigma$ touching $\pa E$ at $x\in\Om$ and  lying on one side with respect to $\pa E$ in a neighborhood of $x$, then $x\in\pa^*E$. 
\end{itemize}
\end{theorem}
Items (i) and (ii) are classical, see for instance Theorems~21.8~and~28.1 in \cite{maggi-book} for (i) and  Theorem~4.7.4 in \cite{Am} for (ii).

Concerning (iii) one can show that under the assumption on $x$ the minimal cone obtained by blowing up $E$ around $x$ is  contained in a half space. For the existence of such a minimal cone see Theorem~28.6 in \cite{maggi-book}. Since any minimal cone contained in a half space is a half space, see for instance \cite[Lemma~3]{DM}, it follows that  $x$ is a regular point.

\vskip 0.2cm
The so-called $\e$-regularity theory for $\Lambda$-minimizers underlying the proof of the above theorem yields  that sequences of $\Lambda$-minimizers $E_h$ converging in $L^1$ to a smooth set $E$ are regular for $h$ large and in fact converge in a stronger sense. More precisely, we have the following result, which is well known to the experts. 

\begin{theorem}\label{th:cicaleo}
Let $E_n$, $E$ be $(\Lambda,r_0)$-minimizers in an open set $\Om\subset\R^N$ such that $E_n\to E$ in $L^1_{loc}(\Om)$. Let $x\in\pa^*E\cap\Om$. Then, up to rotations and translations, there exist a $(N-1)$-dimensional open ball $B'\subset\R^{N-1}$, functions $\varphi_n,\varphi\in W^{2,p}(B')$ for all $p\geq1$, and $r>0$ such that $x\in B'\times(-r,r)$ and for $n$ large
\beq\label{cicaleo1}
\begin{split}
& \pa E_n\cap(B'\times(-r,r))=\{(x',\varphi_n(x')):\,x'\in B'\}, \\
& \pa E\cap(B'\times(-r,r))=\{(x',\varphi(x')):\,x'\in B'\}, \\
& \varphi_n\to\varphi\quad\text{in $C^{1,\alpha}(\overline{B'})$ for some $\alpha\in(0,1)$}\,.
\end{split}
\eeq
Moreover, $H_{\pa E_n}(x',\varphi_n(x'))\stackrel{*}{\wto}H_{\pa E}(x',\varphi(x'))$ in $L^\infty(B')$ and 
thus  $\varphi_n\wto\varphi$ in $W^{2,p}(B')$ for all $p\geq1$.
\end{theorem}
Properties stated in \eqref{cicaleo1} follow from the classical $\e$-regularity theory, see \cite[Th. 1.9]{Tam} (see also  the arguments of Lemma~3.6 in \cite{CL}). The last part of the statement then easily follows from Theorem~\ref{rm:regola}-(ii) combined with the classical Calder\'on-Zygmund estimates.


\section{An estimate of the total positive curvature}\label{sec:tot}

This section is mainly devoted to the proof of Theorem~\ref{th:cgr-main1} and to some applications.

We recall that a set $X\subset \S^{N-1}$ is called {\em spherically convex} (in short {\em convex}) if it is geodesically convex, that is, for any pair of points $x_1$, $x_2\in X$ there exists a distance minimizing geodesic connecting $x_1$ and $x_2$ contained in $X$. 

If $x\in\S^{N-1}$ and $\theta\in (0,\pi)$ we denote by $S_{\theta, x}$ the spherical cap
$$
 S_{\theta, x}:=\{y\in \S^{N-1}:\, x\cdot y\geq \cos\theta\}\,.
$$ 
If $x=e_N$ we shall simply write $S_\theta$ instead of $S_{\theta,e_N}$. 
Note that $S_{\pi-\theta, -x}$ coincides with $(\S^{N-1}\setminus S_{\theta, x})\cup\pa S_{\theta,x}$, where $\pa S_{\theta,x} $ denotes the relative boundary of $S_{\theta, x}$ in $\S^{N-1}$. We recall that 
$$
\H^{N-1}(S_{\theta})=(N-1)\omega_{N-1}\int_0^\theta\sin^{N-2}\sigma\,d\sigma\,,
$$
and $\H^{N-1}(\S^{N-1})=N\omega_N$, where $\omega_N$ is the measure of the unit ball. 

The following lemma extends \cite[Proposition~3.1]{cgr0} to general angles.
\begin{lemma}\label{lm:spheco}
	Let $X\subset\S^{N-1}$ be spherically convex and closed, with  $\H^{N-1}(X)>0$,  let $\theta\in (0,\pi)$ and fix $x\in X$. Then  we have 
	\beq\label{spheco1}
	\H^{N-1}(X\cap S_{\theta, x})\geq \frac{\H^{N-1}(S_\theta)}{N\omega_N} \H^{N-1} (X)\,.
	\eeq
	Moreover, the equality holds  if and only if $-x\in X$. Finally,  given $\theta_0\in(0,\pi)$, for every $\e>0$ there exists $\de>0$, independent of $X$, such that if $\theta\in[\theta_0/2,\theta_0]$, then
  	$$
	\H^{N-1}(X\cap S_{\theta, x})\leq \Big( \frac{\H^{N-1}(S_\theta)}{N\omega_N}+\de\Big) \H^{N-1} (X)\quad\text{implies }\quad \dist(-x, X)\leq \e\,.
	$$
\end{lemma}
\begin{proof} 
We denote by $A$ the spherically convex subset of $S_{\theta, x}$ obtained by taking the union of all the minimal geodesics connecting $x$ with the points of $X\cap \pa S_{\theta, x}$. Let $B:=S_{\theta, x}\setminus A$. 
Similarly denote by $A^-$ the spherically convex subset of $S_{\pi-\theta, -x}$ obtained by taking the union of all the minimal geodesics connecting $-x$ with the points of $X\cap \pa S_{\theta, x}= X\cap \pa S_{\pi-\theta, -x} $, and set $B^-:=S_{\pi-\theta, -x}\setminus A^-$. 

Assume first that $\H^{N-1}(A)>0$. We note that 
$$
\frac{\H^{N-1}(A^-)}{\H^{N-1}(A)}=\frac{\H^{N-1}(S_{\pi-\theta,-x})}{\H^{N-1}(S_{\theta,x})}\,.
$$
Thus, we have
$$
\H^{N-1}(X\cap A)=\H^{N-1}(A)= \frac{\H^{N-1}(S_{\theta,x})}{\H^{N-1}(S_{\pi-\theta,-x})}\H^{N-1}(A^-)\geq \frac{\H^{N-1}(S_{\theta,x})}{\H^{N-1}(S_{\pi-\theta,-x})}\H^{N-1}(X\cap A^-)\,.
$$	 
Note now that $X\cap B^-=\emptyset$. Indeed, if $y\in X\cap B^-$, then the geodesic connecting $y$ to $x$ is contained in $X$ and intersects $\pa S_{\theta,x}$ at a point $z\in X\cap \pa S_{\theta, x}$. It follows in turn that $y $ belongs to the geodesic connecting $z$ with $-x$, and thus $y\in A^-$, which is a contradiction. Therefore,
\beq\label{spheco2}
\begin{split}
\H^{N-1}(X\cap S_{\theta, x}) & =\H^{N-1}(A)+\H^{N-1}(X\cap B)\geq \H^{N-1}(A)\\
&= \frac{\H^{N-1}(S_{\theta,x})}{\H^{N-1}(S_{\pi-\theta,-x})}\H^{N-1}(A^-)\\
&\geq
\frac{\H^{N-1}(S_{\theta,x})}{\H^{N-1}(S_{\pi-\theta,-x})}\H^{N-1}(X\cap S_{\pi-\theta, -x})\,.
	\end{split}
	\eeq
  From this inequality \eqref{spheco1} follows, recalling that $\H^{N-1}(\S^{N-1})=N\omega_N$. 
  
  If instead $\H^{N-1}(A)=0$, then $\H^{N-1}(X\setminus S_{\theta,x})=0$ and thus \eqref{spheco1} holds trivially.
  
  If \eqref{spheco1} holds with the equality, then $\H^{N-1}(A)>0$ and all the inequalities in \eqref{spheco2} are equalities. In particular, 
  $\H^{N-1}(A^-)=\H^{N-1}(X\cap S_{\pi-\theta, -x})>0 $. In turn, by closedness and convexity we deduce 
  {that} $-x\in X$. Conversely, if $-x\in X$ then by spherical convexity we have $A^-= X\cap S_{\pi-\theta, -x} $ and also $X\cap B=\emptyset$ since otherwise any geodesic connecting a point  $y\in X\cap B$ to $-x$ would intersect $\pa S_{\theta, x}\cap X$, thus implying that $y$ belongs to $A$,  a contradiction. Therefore all the inequalities in \eqref{spheco2} are equalities and the conclusion follows. 
  
  To establish the last part, we argue by contradiction assuming that there exist $\e>0$, a sequence of closed spherically convex sets $X_n\ni x$ such that $\H^{N-1}(X_n)>0$ and a sequence $\theta_n\in[\theta_0/2,\theta_0] $ converging to $\theta'$ such that
\beq\label{sphecoabs}
	\H^{N-1}(X_n\cap S_{\theta_n, x})\leq \Big(\frac{\H^{N-1}(S_{\theta_n})}{N\omega_N}+\frac1n\Big) \H^{N-1} (X_n)\quad\text{but }\quad \dist(-x, X_n)\geq \e\,.
\eeq
We denote by $A_n$ and by $A^-_n$ the sets corresponding to $X_n$ and $S_{\theta_n,x}$ defined as above. 
Note that  $X_n=(X_n\cap S_{\theta_n, x})\cup(X_n\cap A^-_n)$. From \eqref{sphecoabs} it follows that $\H^{N-1}(A_n),\,\H^{N-1}(A_n^-)>0$ for $n$ large and
$$
\frac{\H^{N-1}(X_n\cap S_{\theta_n, x})}{\H^{N-1}(X_n\cap A^-_n)}\leq  \frac{\H^{N-1}(S_{\theta_n,x})}{\H^{N-1}(S_{\pi-\theta_n,-x})}+O(\tfrac1n)\,.
$$
Since $\H^{N-1}(X_n\cap S_{\theta_n, x})\geq \H^{N-1}(A_n)\geq  \frac{\H^{N-1}(S_{\theta_n,x})}{\H^{N-1}(S_{\pi-\theta_n,-x})}\H^{N-1}(X_n\cap A^-_n) $, it follows that 
\beq\label{spheco4}
\lim_{n\to\infty} \frac{\H^{N-1}(X_n\cap S_{\theta_n, x})}{\H^{N-1}(X_n\cap S_{\pi-\theta_n, -x})}=\lim_{n\to\infty}\frac{\H^{N-1}(A_n)}{\H^{N-1}(X_n\cap A^-_n)}=\frac{\H^{N-1}(S_{\theta',x})}{\H^{N-1}(S_{\pi-\theta',-x})}\,.
\eeq
Note that  we have
$$
\frac{\H^{N-1}(S_{\theta_n,x})}{\H^{N-1}(S_{\pi-\theta_n,-x})}= \frac{\H^{N-1}(A_n)}{\H^{N-1}( A^-_n)}\to\frac{\H^{N-1}(S_{\theta',x})}{\H^{N-1}(S_{\pi-\theta',-x})}
$$
and thus, from \eqref{spheco4} we get
$$
\lim_{n\to\infty}\frac{\H^{N-1}(A^-_n)}{\H^{N-1}(X_n\cap A^-_n)}=1\,,
$$
which clearly contradicts the fact that by the second inequality in \eqref{sphecoabs} we easily infer that $\H^{N-1}(A_n^-\setminus X_n)\geq C(\e)\H^{N-1}(A_n^-)$, for a positive constant $C(\e)$ depending only in $\e$.
 \end{proof}

Next we adapt to our case \cite[Proposition~4.2]{cgr0}. To this aim we recall some preliminary definitions. 
\begin{definition}\label{normalcones}
Given a set $X\subset\R^N$ and $x\in \R^N$ the {\em unit normal cone} of $X$ at $x$ is the (possibly empty) set defined as
$$
N_xX:=\{\nu\in \S^{N-1}:\, (y-x)\cdot \nu\leq 0\text{ for all }y\in X\}\,.
$$
Any hyperplane passing through $x$ and orthogonal to a direction $\nu\in N_xX$ is called a {\em support hyperplane for $X$ with outward normal $\nu$}. 
In turn, we define the corresponding {\em normal bundle} of $X$ as 
$$ 
NX:=\bigcup_{x\in X}N_xX\,.
$$
Given a map $\sigma: X\to \S^{N-1}$ and $\theta\in (0, \pi)$ we introduce the following {\em restricted normal cone} and {\em restricted normal bundle} respectively as
$$
N_x^{\sigma, \theta} X:=N_xX\cap S_{\theta, \sigma(x)}\qquad\text{and}\qquad
N^{\sigma, \theta} X:=\bigcup_{x\in X} N_x^{\sigma, \theta} X\,.
$$
Moreover, we say that a point $x\in X$ is {\em exposed} if there exists a support hyperplane $\Pi$ passing through $x$ such that $X\cap \Pi=\{x\}$. Finally, we denote by $\wid(X)$ the distance between the closest pair of parallel hyperplanes which contains $X$ in between them, i.e.,
\beq\label{eqwid}
\wid(X)=\inf_{\nu\in\S^{N-1}}\big(\sup\{x\cdot\nu:\,x\in X\}-\inf\{x\cdot\nu:\,x\in X\}\big)\,.
\eeq
\end{definition}
\begin{lemma}\label{lm:wid}
Let $r>0$ and let $X=\{x_1,\dots, x_k\}\subset B_r(0)$. Let $\sigma: X\to \S^{N-1}$ be such that $\sigma(x_i)\in N_{x_i}X$ whenever $N_{x_i}X$ is nonempty. Then
\beq\label{wid1}
\H^{N-1}(N^{\sigma, \theta} X)\geq \H^{N-1}(S_{\theta})\,.
\eeq
Moreover, equality holds in \eqref{wid1} if and only if $X$ lies in a hyperplane $\Pi$ such that $\sigma(x_i)\perp \Pi$ whenever $x_i$ is exposed. Finally, given $\theta_0\in(0,\pi)$, for every $\e>0$ there exists $\de>0$ (depending also on $r>0$ and $\theta_0$, but not on $\sigma$ and not on $X$)  such that if $\theta\in[\theta_0/2,\theta_0]$, then
\beq\label{wid2}
\H^{N-1}(N^{\sigma, \theta}X)\leq \H^{N-1}(S_{\theta})+\de\quad\text{implies}\quad \wid(X)\leq \e\,
\eeq   
and more precisely  there exist an exposed point $x\in X$ and  two parallel hyperplanes orthogonal to $\sigma(x)$ with mutual distance equal to $\e$  such that $X$ lies between them. 
\end{lemma}
\begin{proof}
The proof is essentially the same as for \cite[Proposition~4.2]{cgr0}, using Lemma~\ref{lm:spheco} in place of \cite[Proposition~3.1]{cgr0}. We give the argument for the sake of completeness.  Owing to the compactness of $X$, for every $\nu\in \S^{N-1}$ there exists a support hyperplane to $X$ with outward normal  equal to $\nu$. Thus, $NX=\S^{N-1}$. Observe also that $\nu\in \mathrm{int}_{\S^{N-1}} (N_{x_i}X)$ if and only if  the hyperplane orthogonal to $\nu$ and passing through $x_i$ is a support  hyperplane intersecting $X$ only at $x_i$ (and thus $x_i$ is exposed). In turn, if $i\neq j$ we have
$$
\mathrm{int}_{\S^{N-1}}(N_{x_i}X)\cap \mathrm{int}_{\S^{N-1}}(N_{x_j}X)=\emptyset\,.
$$ 
 Since by \cite[Lemma~4.1]{cgr0} every $N_{x_i}X$ with nonvanishing $\H^{N-1}$-measure is spherically convex, we may invoke Lemma~\ref{lm:spheco} to conclude that 
 \beq\label{wid2.5}
 \begin{split}
  \H^{N-1}(N^{\sigma, \theta}X)=\sum_{i=1}^k\H^{N-1}(N^{\sigma, \theta}_{x_i}X)\geq \frac{\H^{N-1}(S_{\theta})}{N\omega_N} \sum_{i=1}^k\H^{N-1}(N_{x_i}X)=\H^{N-1}(S_{\theta})\,,
  \end{split}
 \eeq
  thus establishing \eqref{wid1}.
  
  If equality holds in \eqref{wid1}, then the above inequality is an equality and in particular 
 \beq\label{wid3}
 \H^{N-1}(N^{\sigma, \theta}_{x_i}X)=\frac{\H^{N-1}(S_{\theta})}{N\omega_N} \H^{N-1}(N_{x_i}X)
   \eeq  
   whenever $\H^{N-1}(N_{x_i}X)>0$, that is whenever $x_i$ is exposed. Therefore, by Lemma~\ref{lm:spheco} $N^{\sigma, \theta}_{x_i}X$ contains both $\sigma (x_i)$ and $-\sigma(x_i)$ and thus $X$ lies in the hyperplane orthogonal to $\sigma(x_i)$ and passing through $x_i$. Conversely, if $X$ lies in a hyperplane orthogonal to $\sigma(x_i)$, for every $x_i$ exposed, then also $-\sigma(x_i)\in N_{x_i}X$ and thus by Lemma~\ref{lm:spheco} \eqref{wid3} holds for all $x_i$ exposed. And thus equality holds also in \eqref{wid2.5}.
   
   To prove \eqref{wid2} and the last part of the lemma,  let  $X_n=\{x^n_1,\dots, x^n_{k_n}\}\subset B_r(0)$ and let  $\sigma_n:X_n\to \S^{N-1}$, with $\sigma_n(x^n_i)\in N_{x^n_i}X_n$ whenever $x_i^n$ is exposed, $\theta_n\in[\theta_0/2,\theta_0]$  be  such that
$$
   \H^{N-1}(N^{\sigma_n, \theta_n}X_n)- \H^{N-1}(S_{\theta_n})\to0\,.
 $$
   Arguing as for \eqref{wid2.5} we then have, in particular, that for every $n\in \N$ there exists $i_n\in \{1, \dots, k_n\}$ such that 
   $$
   \frac{\H^{N-1}(N^{\sigma_n, \theta_n}_{x_{i_n}^n}X_n)}{\H^{N-1}(N_{x^n_{i_n}}X_n)}- \frac{\H^{N-1}(S_{\theta_n})}{N\omega_N}\to0\,.
$$
By Lemma~\ref{lm:spheco} this implies that $\dist(-\sigma_n(x^n_{i_n}), N_{x_{i_n}^n}X_n)\to 0$. From this, owing to the equiboundedness of the $X_n$'s it follows that for every $k\in \N$ and for $n$ large enough  $ X_n$ lies between the two parallel hyperplanes orthogonal to $\sigma_n(x^n_{i_n})$ and passing through the points $x^n_{i_n}$ and   $x^n_{i_n}-\frac1k \sigma_n(x^n_{i_n}) $.  In particular, $\wid(X_n)\to 0$.
\end{proof}




Next proposition extends the previous lemma to the case of a general compact set $X$ and a continuous map $\sigma$.

\begin{proposition}\label{prop:53}
Let $X\subset B_r(0)$ be a compact set.
 Let $\sigma: X\to \S^{N-1}$ be a continuous map such that $\sigma(x)\in N_x X$ for all $x\in X$ such that $N_x X\not=\emptyset$. Then, 
\beq\label{531}
\H^{N-1}(N^{\sigma, \theta}X)\geq \H^{N-1}(S_{\theta})
\eeq
and if equality holds, then $X$ lies in a hyperplane $\Pi$ which is orthogonal to $\sigma(x)$ for some  $x\in X$.
Moreover, given $\theta_0\in(0,\pi)$ and $\e>0$ there exists $\de_0>0$ (depending also on $r$ and $\theta_0$, but not on $\sigma$ and not on $X$)  such that if $\theta\in[\theta_0/2,\theta_0]$, then
\beq\label{540}
\H^{N-1}(N^{\sigma, \theta}X)\leq \H^{N-1}(S_{\theta})+\de_0\quad\text{implies}\quad \wid(X)\leq \e
\eeq   
and more precisely  there exist  $x\in X$ and  two parallel hyperplanes orthogonal to $\sigma(x)$, with mutual distance equal to $\e$  such that $X$ lies between them.
\end{proposition}
\begin{proof} 
Let $\{X_i\}_{i\in \N}$ be an increasing sequence of discrete subsets of $X$ such that $X_i\to X$ in the Hausdorff sense. We claim that 
\beq\label{532}
\Chi{N^{\sigma, \theta}X}\geq\limsup_{i} \Chi{N^{\sigma, \theta}X_i} \quad\text{pointwise in }\S^{N-1}\,.
\eeq  	
To this aim let $\nu\not\in N^{\sigma, \theta}X $ and assume by contradiction that \eqref{532} does not hold at $\nu$ and thus that there exist a subsequence $\{i_n\}$  and points $x_n\in X_{i_n}$ such that $\nu\in N^{\sigma, \theta}_{x_n}X_{i_n}$. Passing to a further (not relabelled) subsequence if needed, we may assume that $x_n\to \bar x\in X$. Observe that  by the continuity of $\sigma(\cdot)$, $\nu\in S_{\theta, \sigma(\bar x)}$. Fix now any  $x\in X$ and due to the Hausdorff convergence find   $y_n\in X_{i_n}$ such that $y_n\to x$. Since for every $n$, $(y_n-x_n)\cdot \nu\leq 0$ passing to the limit we get $(x-\bar x)\cdot \nu\leq 0$. Due to the arbitrariness of $x$, we have shown that $\nu\in N_{\bar x} X$ and thus $\nu\in N^{\sigma, \theta}_{\bar x}X$, a contradiction. 

Using the first part of  Lemma~\ref{lm:wid} (with $X$ replaced by $X_i$), \eqref{532} and Fatou's Lemma we get
$$
\H^{N-1} (N^{\sigma, \theta}X)\geq\limsup_{i} \H^{N-1}({N^{\sigma, \theta}X_i})\geq \liminf_i \H^{N-1}({N^{\sigma, \theta}X_i})\geq \H^{N-1}(S_{\theta})\,.
$$
Assume now that the first inequality \eqref{540} holds for some $\theta\in[\theta_0/2,\theta_0]$, with $\delta_0=\frac\de2$, where $\de$ is the constant provided by Lemma~\ref{lm:wid}. 
Then the previous inequality yields  for $i$ sufficiently large, depending on $\theta$,
$$
\H^{N-1}({N^{\sigma, \theta}X_i}) \leq \H^{N-1}(S_{\theta})+\de
$$
 and thus, thanks to second part of Lemma~\ref{lm:wid} we infer that  there exists  $x_i\in X_i$ and  two parallel hyperplanes orthogonal to $\sigma(x_i)$ with mutual distance equal to $\e$  such that $X_i$ lies between them. By a compactness argument and the continuity of $\sigma$, letting $i\to\infty$ we get that  
 there exist  $x\in X$ and  two parallel hyperplanes orthogonal to $\sigma(x)$ with mutual distance equal to $\e$  such that $X$ lies between them. Thus, in particular $\wid(X)\leq \e$. This establishes \eqref{540}, which in turn, again by a compactness argument and the continuity of $\sigma$, yields the conclusion in the equality case.
\end{proof}

Next we prove a result in the spirit of \cite[Theorem~1.1]{cgr0}. 
In the following  $\cc$, $\Omega$ and $\Sigma$ will be as in Definition~\ref{kappapiunuovo}. Moreover if $x\in\Sigma$ is a point where the tangent hyperplane to $\Sigma$ exists we denote by $\nu_\Sigma(x)$ the normal to this hyperplane pointing outward with respect to $\Omega$. We give the following definition.

\begin{definition}\label{kappapiu}
We  denote by $\Sigma^+$ the set of points in $\Sigma\setminus\cc$ such that there exists a support hyperplane $\Pi_x$ with the property that $\Pi_x\cap \Sigma=\{x\}$.
\end{definition}

We recall the following result, see \cite[Theorem~2.2.9]{S}:
\begin{theorem}\label{th:S}
Let $K\subset\R^N$ be a compact convex set. Then for $\H^{N-1}$-almost every $\nu\in \S^{N-1}$ the support hyperplane for $K$ orthogonal to $\nu$ intersects $K$ at a single point.  
\end{theorem}

\begin{corollary}\label{cor:S}
Let $\cc$ and $\Sigma\subset\R^N$ be as in Definition~\ref{kappapiunuovo}. With the notation above, we have that 
$$
\mathcal K^+(\Sigma)=\H^{N-1}\Big(\bigcup_{x\in\Sigma^+}N_x\Sigma\Big)\,,
$$
where $\mathcal K^+(\Sigma)$ is the total positive curvature defined in Definition~\ref{kappapiunuovo}.
\end{corollary}
\begin{proof}
Let $K$ denote the convex hull of $\Sigma$. By Theorem~\ref{th:S} we have that for $\H^{N-1}$-a.e. direction $\nu\in \bigcup_{x\in\Sigma\setminus\cc} N_x\Sigma$ the corresponding support plane for $K$ intersects $K$ at  a single point that necessarily belongs to $\Sigma\setminus\cc$ and thus to $\Sigma^+$.  
\end{proof}

\begin{proof}[Proof of Theorem~\ref{th:cgr-main1}]
Observe that if  $\Sigma\cap\cc=\emptyset$ then 
$$
\bigcup_{x\in\Sigma\setminus\cc}N_x\Sigma=\S^{N-1}\,,
$$
hence \eqref{cgr-main11} trivially holds.

Hence in the following we may assume that $\Sigma\cap\cc\not=\emptyset$. 
\par\noindent
 We denote by   $\nu_\cc$ the outward normal to $\cc$. 
 We start by proving \eqref{cgr-main11}. Let us define $\sigma: \Sigma\cap\cc\to \S^{N-1}$ as $\sigma(x):=\nu_\cc(x)$. Note that since $\cc$  is convex the direction $\sigma(x)$ belongs to $N_x\cc$ and thus  to $N_x(\Sigma\cap\cc)$ for every $x\in\Sigma\cap\cc$.

Given $\nu\in\S^{N-1}$, we denote by $\nu^\perp$ the hyperplane orthogonal to $\nu$ and passing through the origin and we set
$$
\bar t:=\max\{t\in \R:\, (t\nu+\nu^\perp)\cap \Sigma\neq\emptyset\}\,.
$$
Clearly, by definition  for every $\nu\in \S^{N-1}$ the hyperplane $\bar t\nu+\nu^\perp $ is a support hyperplane for $\Sigma$. Fix $\theta\in(0,\theta_0)$.
We claim that for every $x	\in\Sigma\cap\cc$ 
\beq\label{cgr-main12}
\nu\in N_x(\Sigma\cap\cc)\cap S_{\theta, \sigma(x)}\quad\text{implies}\quad \bar t\nu+\nu^\perp\cap \Sigma\subset \Sigma\setminus\cc\,.
\eeq	
Let $t_0\in \R$ be such that $x+\nu^\perp= t_0\nu+\nu^\perp$ and  observe that since $\nu\cdot \sigma(x)\geq\cos\theta>\cos\theta_0$ then by assumption \eqref{cgr-main10} $\nu\not\in N_x\Sigma$, hence the hyperplane $t_0\nu+\nu^\perp$ enters $\Omega$. Thus it easily follows that $\bar t>t_0$. Let $y\in\bar t\nu+\nu^\perp\cap \Sigma $. Then $y\not\in\Sigma\cap\cc$, since otherwise  this would contradict the fact that $t_0\nu+\nu^\perp$ is a support hyperplane 
for $\Sigma\cap\cc$. This establishes \eqref{cgr-main12}. From \eqref{cgr-main12} it follows that 
\beq\label{cgr-main13}
N^{\sigma, \theta}(\Sigma\cap\cc)\subset\bigcup_{x\in\Sigma\setminus\cc}N_x\Sigma\,.
\eeq
Recall that by  Definition~\ref{kappapiunuovo}
$$
\H^{N-1}\Big(\bigcup_{x\in\Sigma\setminus\cc}N_x\Sigma\Big)=\mathcal K^+(\Sigma)\,.
$$
 Combining the equality above with \eqref{cgr-main13}, the inequality \eqref{cgr-main11} follows from \eqref{531} with $X=\Sigma\cap\cc$, letting $\theta\to\theta_0^-$.  
 
 {Given $\e>0$, let $\delta_0$ be the constant provided by Proposition~\ref{prop:53} and let $\theta\in[\theta_0/2,\theta_0)$ such that
 \beq\label{cgr-main13.1}
 \H^{N-1}(S_{\theta_0})\leq  \H^{N-1}(S_{\theta})+\frac{\delta_0}{2}\,.
 \eeq
Assume that \eqref{cgr-main11.1}  and \eqref{cgr-main11.5} for some   $\delta\in(0,\delta_0/2)$ such that $\cos\theta_0+\delta<\cos\theta$. Then, using the assumption  \eqref{cgr-main11.1}, the same argument as before yields \eqref{cgr-main12}, hence \eqref{cgr-main13}. Thus, from \eqref{cgr-main11.5} and \eqref{cgr-main13.1}} we have in particular
  $$
 \H^{N-1}(N^{\sigma, \theta}(\Sigma\cap\cc))\leq \mathcal K^+(\Sigma)\leq\H^{N-1}(S_{\theta_0})+\delta\leq \H^{N-1}(S_{\theta})+\delta_0\,.
 $$
 The conclusion  follows from Proposition~\ref{prop:53}. \end{proof}
 
 \begin{remark}
 Observe that the equality case in \eqref{cgr-main11} does not imply  $\pa\Om\cap\cc$ lies on a facet of $\cc$. In fact it may happen that $\pa\Om\cap\cc$ is contained in a convex set of Hausdorff dimension strictly less than $N-1$, see Figure~\ref{fig1}.
 \end{remark}
 
 \begin{figure}\label{fig1}
\includegraphics[scale=0.5]{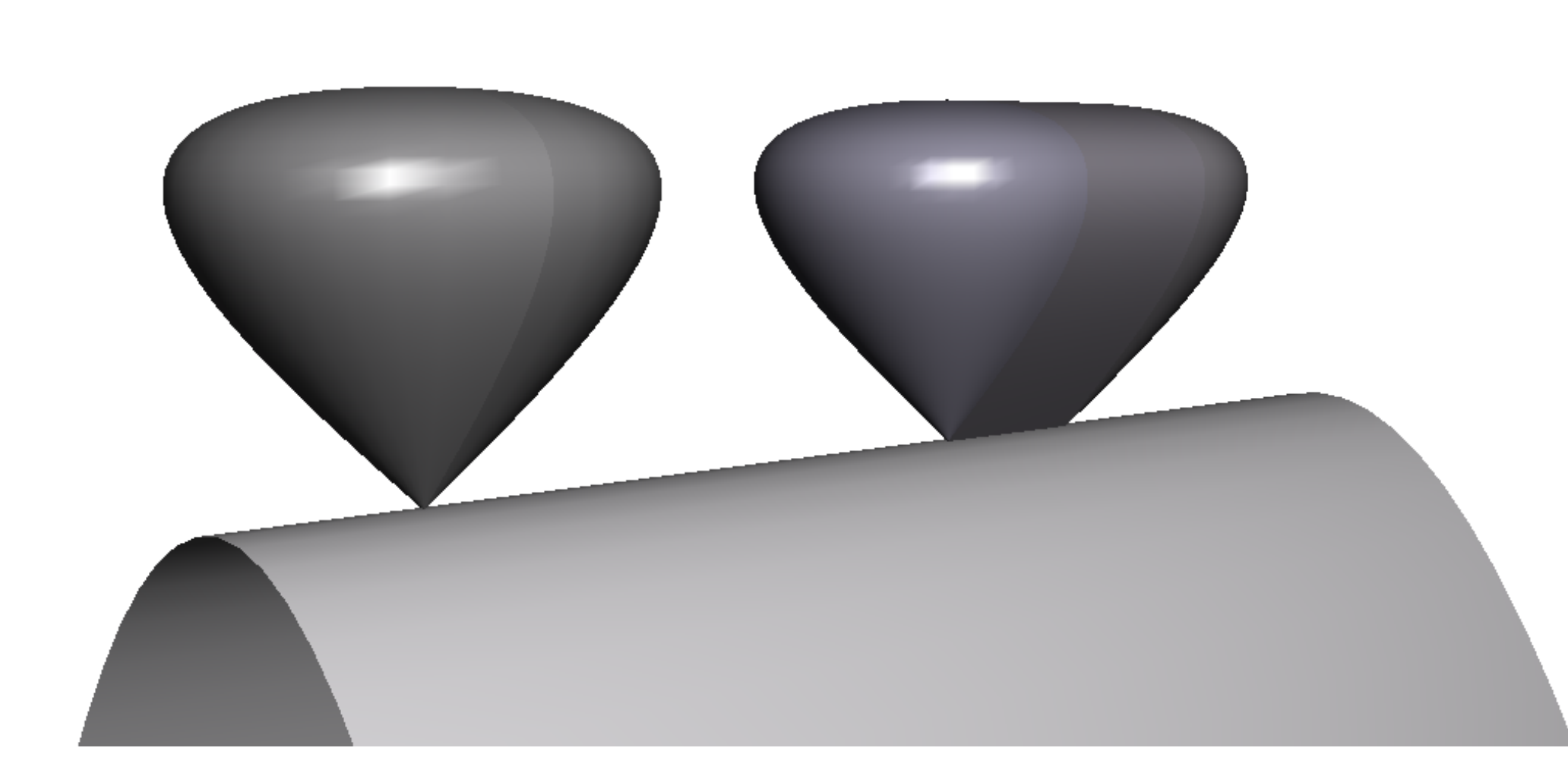} 
 \begin{picture}
 (0,0)\put(-270,80){$\Om_1$}\put(-65,90){$\Om_2$}\put(-150,30){$\cc$}
 \end{picture}
\vskip -.2cm
 \caption{\small Both $\Sigma_1=\overline{\pa\Om_1\setminus\cc}$ and $\Sigma_2=\overline{\pa\Om_2\setminus\cc}$  meet $\cc$ with  contact angle  $\pi/4$ and satisfy the equality in \eqref{cgr-main11} with $\theta_0=\pi/4$. Note that $\pa\Om_1\cap\cc$ is a point and $\pa\Om_2\cap\cc$ is a segment.}
 
\end{figure}
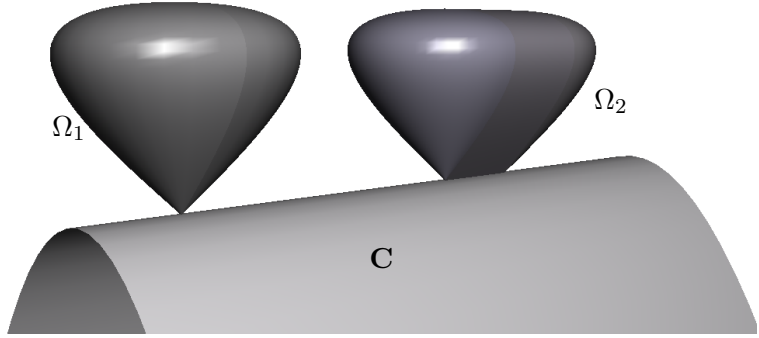

\begin{remark}\label{rm:fessa}
Note that if $x\in\Sigma\cap\cc$ is a point where  $\nu_\Sigma(x)$ exists and belongs to $N_x\Sigma$, then $\nu\cdot\nu_\cc(x)\leq\nu_\Sigma(x)\cdot\nu_\cc(x)$ for every $\nu\in N_x\Sigma$. Therefore in this case it suffices to check \eqref{cgr-main10} for $\nu=\nu_\Sigma(x)$.
\end{remark}

It is well known that for surfaces $\Sigma\subset\R^3$ without boundary the following inequality holds
$$
\int_{\Sigma}|H_\Sigma|^{2}\,d\H^{2}\geq 16\pi\,.
$$
with equality achieved if and only if $\Sigma$ is a sphere. 
We now apply Theorem~\ref{th:cgr-main1} to extend this inequality to the  following  extension of the Willmore energy in $N$-dimensions
$$
\int_{\Sigma\setminus\cc}|H_\Sigma|^{N-1}\,d\H^{N-1}\,,
$$
for $C^{1,1}$ hypersurfaces with boundary supported on convex sets and with contact angle larger than a given $\theta_0\in(0,\pi)$. 
Note  that in the next theorem we do not assume any regularity on the convex set $\cc$.
\begin{theorem}[A  Willmore type inequality.]\label{41}
Let $\cc$, $\Omega$ and $\Sigma$ be as in Definition~\ref{kappapiunuovo} and let $\theta_0\in (0, \pi)$.  Assume  that  $\Sigma\setminus\cc$  is of class $C^{1,1}$. Set $H_\Sigma:=\mathrm{div}_\Sigma{\nu_\Sigma}$ (where $\nu_\Sigma$ is the  unit normal to $\Sigma$ pointing outward with respect to $\Omega$). Assume also 
\beq\label{410}
\nu\cdot\nu'\leq\cos\theta_0 \quad \text{whenever $x\in\Sigma\cap\cc,\,\,\nu\in N_x\Sigma\,\,$ and $\,\,\nu'\in N_x\cc\,.$}
\eeq 
Then,
\beq\label{411}
\int_{\Sigma\setminus\cc}|H_\Sigma|^{N-1}\,d\H^{N-1}\geq (N-1)^{N-1}\H^{N-1}(S_{\theta_0})\,. 
\eeq
Moreover, if equality holds in \eqref{411} and $H_\Sigma\not=0$ a.e., then $\Sigma\setminus\cc$ coincides, up to a rigid motion, with an omothetic of~$S_{\theta_0}$ sitting on a facet of $\cc$.
 \end{theorem}
\begin{proof}
Without loss of generality we may assume that 
$$
\int_{\Sigma\setminus\cc}|H_\Sigma|^{N-1}\,d\H^{N-1}<\infty\,.
$$
Set for any $\eta>0$ sufficiently small $\cc_\eta:=\cc+\overline{B_\eta(0)}$ and  $\Sigma_\eta:=\overline{\pa\Om\setminus \cc_\eta}$. Observe that $\cc_\eta$ satisfies both a outer and inner uniform ball condition and thus is of class $C^{1,1}$, see \cite{MM, D}. Note also that there exists $\theta_\eta \in (0, \theta_0)$ such that 
\beq\label{cgr-main14-}
\nu\cdot\nu_{\cc_\eta}(x)\leq\cos\theta_\eta \quad\text{whenever $x\in\Sigma_\eta\cap\cc_\eta,\,\,\nu\in N_x\Sigma_\eta\,,$}
\eeq
with $\theta_\eta\to \theta_0$ as $\eta\to 0^+$. Indeed, if not, there would exist a sequence $\eta_h\to0$,  a sequence of points $x_h\in\Sigma_{\eta_h}\cap\cc_{\eta_h}$ and a sequence $ \nu_h\in N_{x_h}\Sigma_{\eta_h}$, such that $\nu_h\cdot\nu_{\cc_{\eta_h}}(x_h)\geq\cos\theta'$ for some $\theta'\in(0,\theta_0)$. We may assume that $x_h\to x\in\Sigma\cap\cc$, $\nu_h\to\nu$ and $\nu_{\cc_{\eta_h}}(x_h)\to\nu'$. Clearly $\nu\in N_x\Sigma$, $\nu'\in N_x\cc$ and $\nu\cdot\nu_\cc(x)\geq\cos\theta'$, a contradiction to \eqref{cgr-main10}.

We set 
$$
\widetilde \Sigma=\{x\in\Sigma\setminus\cc:\, N_x\Sigma\not=\emptyset\},\qquad \widetilde \Sigma_\eta=\{x\in\Sigma\setminus\cc_\eta:\, N_x\Sigma_\eta\not=\emptyset\}\,.
$$
We claim that 
\beq\label{cgr-main14}
\Chi{\widetilde\Sigma_\eta}\to \Chi{\widetilde\Sigma} \quad \text{pointwise in $\Sigma\setminus\cc$.} 
\eeq
First of all note that $\widetilde\Sigma\setminus\cc_\eta\subset\widetilde\Sigma_\eta$ for all $\eta$, whence 
$$\Chi{\widetilde\Sigma}=\lim_{\eta\to 0+}\Chi{\widetilde\Sigma\setminus\cc_\eta} \leq\liminf_{\eta\to 0^+} \Chi{\widetilde\Sigma_\eta}\quad\text{pointwise in $\Sigma\setminus\cc$}
$$ 
If otherwise $x\not\in\widetilde\Sigma$, we show that $x\not\in\widetilde\Sigma_\eta$ for $\eta$ small. Indeed, assume by contradiction that   there exist $\nu_h\in N_x(\Sigma_{\eta_h})$, for a sequence $\eta_h\to0$. Then, passing to a subsequence, if needed, $\nu_h\to\nu\in N_x\Sigma$, a contradiction. This proves that
$$
\Chi{\widetilde\Sigma}\geq\limsup_{\eta\to 0^+} \Chi{\widetilde\Sigma_\eta}\quad\text{pointwise in $\Sigma\setminus\cc$}
$$ 
and thus \eqref{cgr-main14} holds.
 
Let $(\Sigma_\eta)^+$ the subset of $\widetilde\Sigma_\eta$ defined as in Definition~\ref{kappapiu} with $\Sigma$ replaced by $\Sigma_\eta$.  
Denote by $K_\Sigma$ the Gaussian curvature of $\Sigma\setminus\cc$ and observe that on $\widetilde\Sigma_\eta$ we have $H_\Sigma\geq0$.
By the arithmetic-geometric mean inequality $(N-1)^{N-1}K_\Sigma\leq H_\Sigma^{N-1}$ on $\widetilde\Sigma_\eta$. Then  by Theorem~\ref{th:cgr-main1}  we  get
\begin{align}
\int_{\Sigma\setminus\cc_\eta}|H_\Sigma|^{N-1}&\,d\H^{N-1}\geq
\int_{\widetilde\Sigma_\eta}H_\Sigma^{N-1}\,d\H^{N-1}\geq(N-1)^{N-1}\int_{\widetilde\Sigma_\eta}K_\Sigma\,d\H^{N-1}\nonumber\\
&\geq(N-1)^{N-1}\int_{(\Sigma_\eta)^+}K_\Sigma\,d\H^{N-1}
=(N-1)^{N-1}\int_{(\Sigma_\eta)^+}\mathrm{det}(D\nu_\Sigma)\,d\H^{N-1}\label{412}\\
&=(N-1)^{N-1}{\mathcal K}^+(\Sigma_\eta)\geq(N-1)^{N-1}\H^{N-1}(S_{\theta_\eta})\,, \nonumber
\end{align}
where in the second equality we have used Corollary~\ref{cor:S} and the area formula, since $\nu_{\Sigma}$ is a  Lipschitz map in a neighborhood of $\Sigma^+$. Then,  letting $\eta\to0$ and recalling \eqref{cgr-main14} and the fact that $\theta_\eta\to\theta_0$, we get
\beq
\begin{split}\label{413}
\int_{\Sigma\setminus\cc}|H_\Sigma|^{N-1}\,d\H^{N-1}&\geq \int_{\widetilde\Sigma}H_\Sigma^{N-1}\,d\H^{N-1}\\
&\geq(N-1)^{N-1}\int_{\widetilde\Sigma}K_\Sigma\,d\H^{N-1}\geq(N-1)^{N-1}\H^{N-1}(S_{\theta_0})\,.
\end{split}
\eeq
In particular \eqref{411} follows.

If equality holds in \eqref{411} holds, from \eqref{412} we have that $\mathcal K^+(\Sigma_\eta)-\H^{N-1}(S_{\theta_\eta})\to0$. In turn  from the second part of Theorem~\ref{th:cgr-main1} we get that $\wid(\Sigma_\eta\cap\cc_\eta)\to0$ and more precisely that  $\Sigma_\eta\cap\cc_\eta$ lies between two parallel  hyperplanes orthogonal to  $\nu_{\cc_\eta}(x_\eta)$  for some $x_\eta\in\Sigma_\eta\cap\cc_\eta$ with mutual distance going to zero. Passing to the limit by a simple compactness argument we infer that  $\Sigma\cap\cc$ lies on a support hyperplane to $\cc$.

Note also that in the equality case, if $H_\Sigma\not=0$ $\H^{N-1}$-a.e., then  \eqref{413} implies that  $\Sigma\setminus\cc=\widetilde\Sigma$. In turn this yields that every $x\in\Sigma\cap\cc$ has a support hyperplane to $\Sigma$.  Moreover,   \eqref{413} yields also that
$H_\Sigma^{N-1}=(N-1)^{N-1}K_\Sigma$.  
 In turn this implies that $\Sigma$ is umbilical and thus, by a classical result, see for instance \cite[Prop.~4, Ch.~3]{dc}, each connected component $\Sigma_i$ of $\Sigma$ is contained in a sphere. Since $\Sigma\cap\cc$ is contained in a hyperplane $\Pi$ tangent to $\cc$,  each  $\Sigma_i$ is either a spherical cap supported on $\Pi$ and satisfying \eqref{410} with $\Sigma$ replaced by $\Sigma_i$, or a sphere not intersecting $\cc$. 
 
 In either case, since \eqref{410} is satisfied at every point in $\Sigma\cap\cc$ (recall that $\Sigma=\widetilde\Sigma$),
 we may apply \eqref{411} to infer that for every connected component $\Sigma_i$ we have
$$
\int_{\Sigma_i\setminus\cc}H_{\Sigma_i}^{N-1}\,d\H^{N-1}\geq (N-1)^{N-1}\H^{N-1}(S_{\theta_0})\,. 
$$
   In particular, since we are in the equality case, there must be only one connected component. Thus   $\Sigma$  is a spherical cap homothetic to $S_{\theta_0}$ up to a rigid motion. Finally $\Sigma\cap\cc$ by convexity must lie on a facet of $\cc$.~\end{proof}

 \section{The equality case in the relative isoperimetric inequality outside a convex set}\label{sc:4}
In this section we give the proof of Theorem~\ref{th:ugual}. 
Throughout this proof  we will denote by $\mathscr H$ the half space
\beq\label{semisp}
\mathscr{H}:=\{(x', x_N)\in \R^N:\, x_N>0\}\,.
\eeq
We will also need the following notions of $(\Lambda,r_0)$-minimizer and  restricted $(\Lambda,r_0)$-minimizer  for the relative perimeter, which extend the standard notion of perimeter $(\Lambda,r_0)$-minimizer recalled in Definition~\ref{def:lambdaminclassic}.

\begin{definition}\label{def:lambdamin}
Let $\cc\subset\R^N$ be a closed convex set with nonempty interior and let $\Lambda,r_0>0$. We say that a set of finite perimeter $E\subset\R^N\setminus\cc$ is a {\em $(\Lambda,r_0)$-minimizer of the relative perimeter}  $P(\cdot;\R^N\setminus\cc)$ if for any $F\subset\R^N\setminus\cc$ such that diam$(E\Delta F)\leq r_0$ we have
$$
P(E;\R^N\setminus\cc)\leq P(F;\R^N\setminus\cc)+\Lambda|E\Delta F|\,.
$$
Moreover, we say that $E\subset\R^N\setminus \cc$ is a {\em restricted  $(\Lambda,r_0)$-minimizer} if the above inequality holds for every set  $F\subset\R^N\setminus\cc$ such that diam$(E\Delta F)\leq r_0$  and $\pa^*F\cap\cc\subset\pa^*E\cap\cc$ up to a $\H^{N-1}$-negligible set.
\end{definition}

\begin{proof}[Proof of Theorem~\ref{th:ugual}]
Let $m_0>0$ be a given mass and let $\Om_0$ be a minimizer of the perimeter outside $\cc$ such that $|\Om_0|=m_0$ and
\beq\label{ugual0}
P(\Om_0;\R^N\setminus\cc)=N\Big(\frac{\omega_N}{2}\Big)^{\frac{1}{N}}m_0^{\frac{N-1}{N}}\,.
\eeq
Since $\Om_0$ solves the isoperimetric problem we have that $\Omega_0$ is a $(\Lambda_0,r_0)$-minimizer of the relative perimeter in $\R^N\setminus\cc$, see Definition~\ref{def:lambdamin}, 
for some $\Lambda_0, r_0>0$, depending on $\Om_0$, see for instance the argument of \cite[Example 21.3]{maggi-book}.\footnote{Note that in \cite[Example 21.3]{maggi-book} it is proved that a mass constrained minimizer $E$ of the relative perimeter in an open set $A$ is a perimeter $(\Lambda_0,r_0)$-minimizer  in $A$ according to Definition~\ref{def:lambdaminclassic}. However an inspection of the proof shows that $E$ is also a $(\Lambda_0,r_0)$-minimizer of the relative perimeter $P(\cdot,A)$ according to Definition~\ref{def:lambdamin}.}  In turn by Proposition~\ref{densityestimates} $\Om_0$ satisfies uniform volume density estimates and thus it easily follows that $\Om_0$ is bounded.

We fix a sufficiently large ball $B_R(0)$ containing $\overline{\Om_0}$. Note that by standard argument, see also the argument of Step 1 below, $\Om_0$ solves the following penalized minimum problem
$$
\min\{P(E;\R^N\setminus\cc)+\Lambda_0||E|-m|:\,\,E\subset B_R\setminus\cc\}\,,
$$
for a possibly larger $\Lambda_0$. In particular   we have
\beq\label{ugual0.1}
P(\Omega_0;\R^N\setminus\cc)\leq P(E;\R^N\setminus\cc)+\Lambda_0|\Om_0\Delta E|\qquad\text{for all $E\subset B_R\setminus\cc$.}
\eeq
Since in the remaining part of the proof we will always work inside $B_R$, up to replacing $\cc$ with $\cc\cap\overline{B_R}$, we may assume without loss of generality that $\cc$ is bounded.

Observe that by Theorem~\ref{rm:regola} we may assume that $\Om_0$ is an open set and that $\partial\Om_0\setminus\cc$ coincides with the reduced boundary $\partial^*\Om_0\setminus\cc$ up to an $\H^{N-1}$-negligible set.
Let us show that  $\Om_0$ is connected. Indeed, if otherwise $\Omega_0=\Omega_1\cup\Om_2$, with $\Om_1$ and $\Om_2$ open, $\Om_1$ a connected component of $\Om_0$ with $0<|\Om_1|<m_0$, we have by Theorem~\ref{th:isoperim}
\begin{align*}
P(\Om_0;\R^N\setminus\cc)&
=P(\Om_1;\R^N\setminus\cc)+P(\Om_2;\R^N\setminus\cc)\\
&\geq N\Big(\frac{\omega_N}{2}\Big)^{\frac{1}{N}}|\Om_1|^{\frac{N-1}{N}}+N\Big(\frac{\omega_N}{2}\Big)^{\frac{1}{N}}|\Om_2|^{\frac{N-1}{N}}>N\Big(\frac{\omega_N}{2}\Big)^{\frac{1}{N}}m_0^{\frac{N-1}{N}}\,,
\end{align*}
which is a contradiction to \eqref{ugual0}.

 For every $\eta\geq0$ we set $\cc_\eta=\cc+\overline{B_\eta(0)}$ and, for $\eta\in[0,\bar\eta]$  we set $m_\eta:=|\Om_0\setminus\cc_\eta|$, where $\bar\eta>0$ is such that  $|\Om_0\setminus\cc_{\bar\eta}|>0$. Correspondingly, we set for $m\in(0,m_\eta]$
\beq\label{ieta}
I_\eta(m)=\min\{P(E;\R^N\setminus\cc_\eta):\,\,E\subset \Om_0\setminus\cc_\eta,\,\,|E|=m\}
\eeq
and denote by $\Om_{\eta,m}$ any minimizer of the above problem. {Note that $\Om_{0,m_0}=\Om_0$. Observe also that
\beq\label{supsup}
\sup_{\eta\in[0,\bar\eta]}\sup_{m\in(0,m_\eta)}P(\Om_{\eta,m})<\infty\,.
\eeq
Indeed, given $\eta\in[0,\bar\eta]$ and $0<m\leq m_\eta$ there exists $\eta'\geq\eta$ such that $|\Om_0\setminus\cc_{\eta'}|=m$. Thus
\begin{align*}
P(\Om_{\eta,m})&\leq P(\Om_{\eta,m};\R^N\setminus\cc_\eta)+P(\cc_\eta;B_R)\leq P(\Om_0\setminus\cc_{\eta'})+P(\cc_\eta;B_R)\\
&\leq P(\Om_0;\R^N\setminus\cc)+2\sup_{s\geq0}P(\cc_s;B_R)\leq P(\Om_0;\R^N\setminus\cc)+2N\omega_NR^{N-1}\,.
\end{align*}
}

Let us fix  $m', m''\in(0,m_0)$, with $m'<m''$. 
We claim that there exists $\tilde\eta\in(0,\bar\eta]$ such that 
\beq\label{lunga}
\text{if $\eta\in[0,\tilde\eta]$ and $U$ is a connected component of $\Om_0\setminus\cc_\eta$, then $|U|\not\in[m',m'']$.}
\eeq
Note that this property implies in  particular  that
\beq\label{ugual1.05}
\pa\Om_{\eta,m}\cap(\Om_0\setminus\cc_\eta)\not=\emptyset\qquad\text{for all $\eta\in[0,\tilde\eta]$ and $m\in[m',m'']$.}
\eeq
 To prove \eqref{lunga} we fix $x_0\in\Om_0$ and for every $\eta$ we denote by $U_\eta$ the connected component of $\Omega_0\setminus\cc_\eta$ containing $x_0$. Note that $U_\eta$ increases as $\eta$ becomes smaller. Given any other point $x\in\Omega_0$ there exists a path connecting $x_0$ and $x$ contained in $\Omega_0$, thus $x\in U_\eta$ for $\eta$ small enough. Hence $|U_\eta|\to m_0$, and  the claim follows.

We split the remaining part of the proof in several steps. Some of the long technical claims contained in these steps will be proved in Appendix~B so as not to break the line of reasoning.

\par\noindent
{\bf Step 1.} ({\it Equivalence with a volume penalized problem}). 
Fix $0<m'<m''<m_0$ and let $0<\tilde\eta\leq\bar\eta$ be as in \eqref{lunga}. 
We claim that there exists $\Lambda'>0$  with the following property:  for every $\eta\in[0,\tilde\eta]$ and $m\in [m',m'']$ we have that $\Om_{\eta,m}$ is a minimizer of the following problem
\beq\label{ugual1}
\min\{P(E;\R^N\setminus\cc_\eta)+\Lambda'||E|-m|:\,\,E\subset \Om_0\setminus\cc_\eta\}\,.
\eeq
The proof of this claim will be given in the Appendix~B.

\par\noindent
{\bf Step 2.} ({\it $\Om_{\eta,m}$ is a restricted $\Lambda$-minimizer}). Fix $0<m'<m''<m_0$, let $0<\tilde\eta\leq\bar\eta$ be as in \eqref{lunga} and set $\Lambda=\max\{\Lambda',\Lambda_0\}$,   where $\Lambda'$ is as in Step 1. We claim that for every $\eta\in[0,\tilde\eta]$ and $m\in [m',m'']$, $\Om_{\eta,m}$ is a restricted $\Lambda$-minimizer
 under the constraint that $\partial^*E\cap\cc_\eta\subset\partial^*(\Omega_0\setminus\cc_\eta)\cap\cc_\eta$. 
 More precisely, for every set of finite perimeter $E\subset B_R(0)\setminus\cc_\eta$ such that $\partial^*E\cap\cc_\eta\subset\partial^*(\Omega_0\setminus\cc_\eta)\cap\cc_\eta$ up to a $\H^{N-1}$-negligible set
\beq\label{lambdaeta}
P(\Om_{\eta,m};\R^N\setminus\cc_\eta)\leq P(E;\R^N\setminus\cc_\eta)+\Lambda|\Om_{\eta,m}\Delta E|\,.
\eeq
In particular $\Om_{\eta,m}$ is a restricted $(\Lambda,r_0)$-minimizer according to Definition~\ref{def:lambdamin}, choosing for instance $r_0:=$dist$(\Om_0,\pa B_R(0))$. 

 Given $E$ as above, from Step 1 we get
\begin{align}
P(\Om_{\eta,m};\R^N\setminus\cc_\eta)
&\leq P(E\cap\Om_0;\R^N\setminus\cc_\eta)+\Lambda'||E\cap\Om_0|-|\Om_{\eta,m}|| \nonumber\\
&=\H^{N-1}(\partial^*E\cap(\Om_0\setminus\cc_\eta))+\H^{N-1}(\pa^*E\cap\pa^*\Om_0\cap\{\nu_E=\nu_{\Om_0}\}\setminus\cc_\eta)\label{ugual1.5}\\
&\quad+\H^{N-1}(\pa^*\Om_0\cap E^{(1)})+\Lambda'||E\cap\Om_0|-|\Om_{\eta,m}||\,.\nonumber
\end{align}
Then, using \eqref{ugual0.1} and the condition $\partial^*E\cap\cc_\eta\subset\partial^*(\Omega_0\setminus\cc_\eta)\cap\cc_\eta$, we have
\begin{align*}
&\H^{N-1}((\pa^*\Om_0\cap\cc_\eta)\setminus\cc)+\H^{N-1}(\pa^*\Om_0\setminus\cc_\eta)=P(\Om_0;\R^N\setminus\cc)\\
&\leq P(E\cup\Om_0;\R^N\setminus\cc)+\Lambda_0|E\setminus\Om_0|\\
&=\H^{N-1}((\pa^*\Om_0\cap\cc_\eta)\setminus\cc)+\H^{N-1}((\pa^*\Om_0\cap E^{(0)})\setminus\cc_\eta)\\
&\qquad+\H^{N-1}(\pa^*E\cap\pa^*\Om_0\cap\{\nu_E=\nu_{\Om_0}\}\setminus\cc_\eta)+\H^{N-1}(\pa^*E\setminus{\overline\Om_0})+\Lambda_0|E\setminus\Om_0|\,.
\end{align*}
Simplifying the above inequality, we get
$$
\H^{N-1}(\pa^*\Om_0\cap E^{(1)})+\H^{N-1}(\pa^*E\cap\pa^*\Om_0\cap\{\nu_E=-\nu_{\Om_0}\})\leq
\H^{N-1}(\pa^*E\setminus{\overline\Om_0})+\Lambda_0|E\setminus\Om_0|\,.
$$
Combining this inequality with \eqref{ugual1.5} we conclude that
\begin{align*}
P(\Om_{\eta,m};\R^N\setminus\cc_\eta)
&\leq \H^{N-1}(\partial^*E\cap(\Om_0\setminus\cc_\eta))+\H^{N-1}(\pa^*E\cap\pa^*\Om_0\cap\{\nu_E=\nu_{\Om_0}\}\setminus\cc_\eta)\\
&\quad+\H^{N-1}(\pa^*E\setminus{\overline\Om_0})+\Lambda'||E\cap\Om_0|-|\Om_{\eta,m}||+\Lambda_0|E\setminus\Om_0|\\
&\leq P(E;\R^N\setminus\cc_\eta)+\max\{\Lambda',\Lambda_0\}|E\Delta\Om_{\eta,m}|
\end{align*}
so that the claim is proven.

\par\noindent
{\bf Step 3.} ({\it Monotonicity and Lipschitz equicontinuity of the isoperimetric profiles}). We claim that $I_\eta$ (see  \eqref{ieta}) is strictly increasing in $[0,m_\eta]$ for all $\eta\in[0,\bar\eta]$. Moreover,  for any fixed $0<m'<m''<m_0$  and for  $0<\tilde\eta\leq\bar\eta$ be as in \eqref{lunga}, 
we claim that   for $\eta\in[0,\tilde\eta]$, $I_\eta$  is $\Lambda'$-Lipschitz in $[m',m'']$, where $\Lambda'$ is as in Step 2.

We postpone the proof to Appendix~B.

\par\noindent
{\bf Step 4.} ({\it A formula for $I'_\eta$}). Fix $0<m'<m''<m_0$ and let $0<\tilde\eta\leq\bar\eta$ be as in \eqref{lunga}. For $m\in[m',m'']$ and $\eta\in[0,\tilde\eta]$  we set $\Sigma_{\eta,m}:=\overline{\pa\Om_{\eta,m}\setminus\cc_\eta}$ and denote by $\Sigma^*_{\eta,m}$ the regular free part of $\Sigma_{\eta,m}$, that is $\Sigma^*_{\eta,m}:=\pa^*\Om_{\eta,m}\setminus(\pa\Om_0\cup\cc_\eta)$. Observe that by \eqref{ugual1.05} $\Sigma^*_{\eta,m}$ is nonempty. We recall that by a standard first variation argument $\Sigma^*_{\eta,m}$ is a constant mean curvature manifold. We denote by $H_{\Sigma^*_{\eta,m}}$ such a mean curvature.

We claim that at any point  $m\in(m',m'')$ of differentiability for  $I_\eta$, $\eta\in[0,\tilde\eta]$, we have
\beq\label{isoprof1}
I'_\eta(m)=H_{\Sigma^*_{\eta,m}}\,.
\eeq
 To this end we fix $x\in\Sigma^*_{\eta,m}$ and a ball $B_r(x)\subset\!\subset\Om_0\setminus\cc_\eta$ such that $\Sigma^*_{\eta,m}\cap B_r(x)=\pa\Om_{\eta,m}\cap B_r(x)$. Let $X$ be a smooth vector field compactly supported in $B_r(x)$ such that 
$$
\int_{\Sigma^*_{\eta,m}}X\cdot\nu_{\Om_{\eta,m}}\,d\H^{N-1}\not=0\,.
$$ 
 Consider now the flow associated with $X$, that is the solution in $\R^N\times\R$ of
\[
\begin{cases}
\frac{\pa\Phi}{\pa t}(x,t)=X(\Phi(x,t)) \cr
\Phi(x,0)=x
\end{cases}
\]
and set $\Om_{\eta,m}(t):=\Phi(\Om_{\eta,m},t)$. Clearly, $P(\Om_{\eta,m}(t));\R^N\setminus\cc_\eta)\geq I_\eta(|\Om_{\eta,m}(t)|)$, with the equality at $t=0$. Therefore
$$
\frac{d}{dt}\Big(P(\Om_{\eta,m}(t));\R^N\setminus\cc_\eta)\Big)_{\big|_{t=0}}=\frac{d}{dt}\Big(I_\eta(|\Om_{\eta,m}(t)|)\Big)_{\big|_{t=0}}\,.
$$
Note that
\begin{align*}
&\frac{d}{dt}\Big(P(\Om_{\eta,m}(t));\R^N\setminus\cc_\eta\Big)_{\big|_{t=0}}
=H_{\Sigma^*_{\eta,m}}\int_{\Sigma^*_{\eta,m}}X\cdot\nu_{\Om_{\eta,m}}\,d\H^{N-1}\,, \\
&\frac{d}{dt}\Big(I_\eta(|\Om_{\eta,m}(t)|)\Big)_{\big|_{t=0}}=I'_\eta(m)\frac{d}{dt}\big(|\Om_{\eta,m}(t)|\big)_{\big|_{t=0}}=
I'_\eta(m)\int_{\Sigma^*_{\eta,m}}X\cdot\nu_{\Om_{\eta,m}}\,d\H^{N-1}\,,
\end{align*}
where we have used the well known formulas for the first variation of the perimeter and the volume, see for instance \cite[Chap. 17]{maggi-book}. Thus \eqref{isoprof1} follows.
\par\noindent
{\bf Step 5.} ({\it A weak Young's law}). Fix $0<m'<m''<m_0$ and let $0<\tilde\eta\leq\bar\eta$ be as in \eqref{lunga}. We claim that the if $\eta\in[0,\tilde\eta]$ and $m\in[m',m'']$, the following weak Young's law holds:
\beq\label{wyl1}
\nu\cdot\nu_{\cc_\eta}(x)\leq0\quad \text{whenever $x\in\Sigma_{\eta,m}\cap\cc_{\eta}$ and $\nu\in N_x\Sigma_{\eta,m}$.}
\eeq
Let $x\in\Sigma_{\eta,m}\cap\cc_{\eta}$ and $\nu\in N_x\Sigma_{\eta,m}$. 
Without loss of generality, by rotating the coordinate system if needed, we may assume that $x=0$, $\nu_{\cc_\eta}(0)=e_N$ and $\nu=(\nu_1,0,\dots,0,\nu_N)$ with $\nu_1\leq0$. Note that \eqref{wyl1} will be proven if we show that 
\beq\label{nu1}
\text{$\nu_N\geq0$ implies that $\nu_N=0\,.$}
\eeq
 Set $E_h=\frac{1}{h}\Om_{\eta,m}$, $h\in\N$ and $\cc_{\eta,h}=\frac{1}{h}\cc_\eta$ and observe that, since $\nu_1\leq0$ and $\nu_N\geq0$, $E_h\subset\{x_1\geq0\}$.  Note also that by \eqref{lambdaeta} we have that
\beq\label{lambdaetah}
P(E_h;\R^N\setminus\cc_{\eta,h})\leq P( G;\R^N\setminus\cc_{\eta,h})+\frac{1}{h}\Lambda|E_h\Delta G|
\eeq
for all sets $G\subset B_{hR}(0)\setminus\cc_{\eta,h}$ such that $\partial^*G\cap\cc_{\eta,h}\subset \partial^*E_h\cap\cc_{\eta,h}$ up to a $\H^{N-1}$-negligible set.
Using the density estimate proved in Proposition~\ref{densityestimates} and passing possibly to a not relabelled subsequence we may assume that  $E_h$ converge in $L^1_{loc}(\R^N)$ to some set $E\subset\mathscr{H}\cap\{x_1>0\}$ (see \eqref{semisp}) of locally finite perimeter and that $\mu_{E_h}\wtos \mu_E$ as Radon measures in $\R^N$, {see \eqref{gg} for the definition of $\mu_E$}. Finally, given $r>0$, from the volume density estimate in Proposition~\ref{densityestimates} we get that for $h$ large enough $|E_h\cap B_r(0)|\geq cr^N$ and thus, passing to the limit, we have $|E\cap B_r(0)|\geq cr^N$ for all $r>0$. This in turn implies that {$0\in\pa^eE\subset\pa E$}.
Since each $E_h$ is a $\frac{\Lambda}{h}$-minimizer, by  Theorem~\ref{th:compact} we have that $E$ is a $0$-minimizer that is
\beq\label{calim2}
P(E;B_r(x_0))\leq P(F; B_r(x_0))\text{ for any $F, B_r(x_0)$ s.t. $E\Delta F\subset\!\subset B_r(x_0)\subset\subset\mathscr H$\,.}
\eeq

We claim  that also the minimality with respect to inner perturbations passes to the limit. More precisely
we want to show that $E$ satisfies the following minimality property: for any cube $Q_r(0)=(-r,r)^N$
and any open set with Lipschitz boundary $V\subset\subset Q_r(0)$   
\beq\label{calim4}
\H^{N-1}(\pa E\cap\pa V\cap\mathscr H)=0\quad\text{implies}\quad
P(E;\mathscr H\cap Q_r(0))\leq P(E\setminus V;\mathscr H\cap Q_r(0))\,.
\eeq
We postpone the proof of this claim to Appendix B.

We now denote by $\widehat E=E\cup R(E)$ where $R$ denotes the reflection map $R(x',x_N)=(x',-x_N)$. From \eqref{calim4} one can easily check that given an open set with Lipschitz boundary $V\subset\!\subset Q_r(0)$ such that $\H^{N-1}(\pa\widehat E\cap\pa V)=0$ we have 
$$
P(\widehat E; Q_r(0))\leq P(\widehat E\setminus V; Q_r(0))\,.
$$
We claim that the connected component $\Gamma$ of $\pa\widehat E$ containing $0$ coincides with $\{x_1=0\}$. In turn this implies \eqref{nu1}.

To see this assume first that  $\Gamma$ intersects $\{x_1=0\}\setminus\{x_N=0\}$ at some point $x_0$. Then, by Theorem~\ref{rm:regola}-(iii)  $\Gamma$ is a smooth minimal surface in a neighborhood of $x_0$. In turn,  by the Strong Maximum Principle Theorem~\ref{th:SMP} it coincides with the hyperplane $\{x_1=0\}$ in a neighborhood of $x_0$. The same argument shows that $\Gamma\cap\{x_1=0\}$ is both relatively closed and open in $\{x_1=0\}$ and therefore $\Gamma=\{x_1=0\}$.

Observe that $\pa\widehat E\cap\{x_1=0\}\subset\{x_1=0\}\cap\{x_N=0\}$ and thus in particular $\H^{N-1}(\pa \widehat E\cap\{x_1=0\})=0$. We may then apply Lemma~\ref{th:maxprin} to conclude that $0\not\in\pa\widehat E$, thus getting a contradiction.

\par\noindent
{\bf Step 6.} ({\it Convergence of the isoperimetric profiles}). We claim that 
\beq\label{ugual2}
\lim_{\eta\to0}I_\eta(m)=I_0(m)\,\,\text{for all $m\in[0,m_0)\quad$ and}\quad \lim_{\eta\to0}I_\eta(m_\eta)=I_0(m_0)\,.
\eeq
Let $\eta_n$ be a sequence converging to zero such that $I_{\eta_n}(m)\to\liminf_{\eta\to0}I_\eta(m)$. Since the perimeters of $\Om_{\eta_n,m}$ are equibounded, see \eqref{supsup},  up to a subsequence we may assume that $\Om_{\eta_n,m}$ converge in $L^1$ to a set of finite perimeter $E\subset\Om_0$ with $|E|=m$. Thus, by lower semicontinuity,
\beq\label{ugual3}
I_0(m)\leq P(E;\R^N\setminus\cc)\leq\liminf_{n}P(\Om_{\eta_n,m},\R^N\setminus\cc_{\eta_n})=\liminf_{\eta\to0}I_\eta(m)\,.
\eeq
Recall that $\Om_{0,m}$ denotes a minimizer for the problem defining $I_0(m)$. Since
$$
I_\eta(m-|\Om_{0,m}\cap\cc_\eta|)\leq P(\Om_{0,m};\R^N\setminus\cc_\eta)\leq I_0(m)\,,
$$
using the equilipschitz continuity of $I_\eta$ proved in Step 3, by letting $\eta$ tend to $0$ in the previous inequality and recalling \eqref{ugual3} we obtain the first equality in \eqref{ugual2}. The second one follows simply from the fact that  $\Om_{\eta,m_\eta}=\Om_0\setminus\cc_\eta$.

Note that the above argument shows in particular that if $m\in(0,m_0)$, $\eta_n\to0$ and $\Om_{\eta_n,m}$ is a sequence converging in $L^1$ to a set $E$, then $E$ coincides with a minimizer $\Om_{0,m}$.
\par\noindent
{\bf Step 7.} ($I_0=I_{\mathscr H}$). We set 
\beq\label{ugual4}
I_{\mathscr H}(m)=N\Big(\frac{\omega_N}{2}\Big)^{\frac{1}{N}}m^{\frac{N-1}{N}}\,,
\eeq 
that is the isoperimetric profile of half spaces. We claim that 
\beq\label{ugual4.1}
I_0(m)=I_{\mathscr H}(m)\qquad \text{for all $m\in[0,m_0]$}\,.
\eeq
To this end we fix $0<m'<m''<m_0$  and let $0<\tilde\eta\leq\bar\eta$ be as in \eqref{lunga}. Recall that by Step 2 for all $\eta\in[0,\tilde\eta]$,  $\Om_{\eta,m}$ is a restricted $(\Lambda,r_0)$-minimizer for all $m\in[m',m'']$. We claim that for any such $\eta$ if
 $x_0\in\Sigma^+_{\eta,m}$ then $\Sigma_{\eta,m}$ is of class $C^{1,1}$ in a neighborhood of $x_0$. Here $\Sigma^+_{\eta,m}$ is defined as in Definition~\ref{kappapiu} with $\Sigma$ and $\cc$ replaced by $\Sigma_{\eta,m}$ and $\cc_\eta$. 
Indeed, observe first that if $x_0\in\Sigma^+_{\eta,m}$ then by Theorem~\ref{rm:regola} $\Sigma_{\eta,m}$ is of class $C^{1,\alpha}$ in a neighborhood of $x_0$. Moreover, if $x_0\in\Omega_0$ then, since $H_{\Sigma_{\eta,m}}$ is constant in a neighborhood of $x_0$, we have that in fact $\Sigma_{\eta,m}$ is analytic in such a neighborhood. 

If instead $x_0\in\pa\Om_0$, since $\Om_0$ is a $(\Lambda,r_0)$-minimizer and $\pa\Om_0$ lies on one side with respect to $\Sigma_{\eta,m}$ which is of class $C^{1,\alpha}$ in a neighborhood of $x_0$, again by Theorem~\ref{rm:regola} we infer that $\pa\Om_0$ is of class $C^{1,\alpha}$, hence analytic in a neighborhood of $x_0$. The claim then follows from Proposition~\ref{strezi}.

To prove \eqref{ugual4.1} observe that the very same argument  of \eqref{412} (with $\widetilde\Sigma_\eta$ replaced by $\Sigma_{\eta,m}^+$ and $\theta_\eta$ replaced by $\pi/2$) yields that
\beq\label{ugual6}
\int_{\Sigma_{\eta,m}^+\setminus\cc_\eta}H_{\Sigma_{\eta,m}}^{N-1}\,d\H^{N-1}\geq (N-1)^{N-1}N\frac{\omega_N}{2}\,. 
\eeq 
Indeed this argument only requires that $\Sigma_{\eta,m}$ is of class $C^{1,1}$ in a neighborhood of $\Sigma_{\eta,m}^+$ and that \eqref{cgr-main14-} holds. Recall that the latter condition with $\theta_\eta=\pi/2$ is ensured by Step 5. Observe also that if $\Sigma_{\eta,m}^+$ intersects $\pa\Om_0$ in a set of positive $\H^{N-1}$ measure then for $\H^{N-1}$-a.e. $x$ on such a set 
\beq\label{ugual6.2}
H_{\Sigma_{\eta,m}}(x)=H_{\pa\Om_0}\leq H_{\Sigma_{\eta,m}^*}
\eeq 
where $\Sigma_{\eta,m}^*$ is the regular free part defined in Step 4 and the inequality follows from Proposition~\ref{strezi}. Here, with a slight abuse of notation, we denote by $H_{\pa\Om_0}$ the constant curvature of $\pa^*\Om_0\setminus\cc$. Therefore the previous inequality, \eqref{ugual6} and \eqref{isoprof1} imply in particular that for a.e. $m\in(m',m'')$  and for all $\eta\in[0,\tilde\eta]$ 
\beq\label{ugual4.2}
\begin{split}
I_\eta(m)(I'_\eta(m))^{N-1}&=P(\Om_{\eta,m};\R^N\setminus\cc_\eta) H_{\Sigma_{\eta,m}^*}^{N-1}\\
&\geq (N-1)^{N-1}N\frac{\omega_N}{2}=I_{\mathscr H}(m)(I_{\mathscr H}'(m))^{N-1}\,,
\end{split}
\eeq
where the last equality follows from \eqref{ugual4}.  Recalling that   $I_\eta$ is  Lipschitz in $[m',m'']$ and thus absolutely continuous, raising the above inequality to the power $\frac{1}{N-1}$ and integrating  in $[m,m'']$, for any $m\in(m',m'')$ we get
$$
I_\eta(m'')^{\frac{N}{N-1}}-I_\eta(m)^{\frac{N}{N-1}}
\geq I_{\mathscr H}(m'')^{\frac{N}{N-1}}-I_{\mathscr H}(m)^{\frac{N}{N-1}}
$$
for all $\eta\in[0,\tilde\eta]$.
Passing to the limit as $\eta\to0$ and using Step 6 we get
\beq\label{ugual5}
I_0(m'')^{\frac{N}{N-1}}-I_0(m)^{\frac{N}{N-1}}
\geq I_{\mathscr H}(m'')^{\frac{N}{N-1}}-I_{\mathscr H}(m)^{\frac{N}{N-1}}
\eeq
for all $0<m<m''<m_0$.
Observe now that $\lim_{m''\to m_0}I_0(m'')=I_0(m_0)$ (this follows by a simple semicontinuity argument and by the fact that $I_0$ is increasing). Thus, passing to the limit in \eqref{ugual5} as $m''\to m_0$,
 recalling that by assumption $I_0(m_0)=I_{\mathscr H}(m_0)$ and that by Theorem~\ref{th:isoperim} $I_0(m)\geq I_{\mathscr H}(m)$, we get $I_0(m)= I_{\mathscr H}(m)$ for all $m\in(0,m_0)$, as claimed. 
\par\noindent
{\bf Step 8.} ({\it $\pa\Om_0\cap\cc$ is flat}). In this step we prove that  $\pa\Om_0\cap\cc$ lies on a  hyperplane~$\Pi$. 

To this aim we start by showing that
\beq\label{ugual6.1}
\big(I_{\eta}^{\frac{N}{N-1}}\big)'\to \big(I_{\mathscr H}^{\frac{N}{N-1}}\big)'\qquad\text{in $L^1_{loc}(0,m_0)$}\,.
\eeq
  Indeed, given $0<m'<m''<m_0$ from \eqref{ugual4.2} and the fact that $I_\eta\to I_{\mathscr H}$, we have that for a.e. $m\in(m',m'')$ and $\eta\in[0,\tilde\eta]$, $(I_{\eta}^{\frac{N}{N-1}}\big)'(m) \geq\big(I_{\mathscr H}^{\frac{N}{N-1}}\big)'(m)$ and 
$$
\int_{m'}^{m''}\big(I_{\eta}^{\frac{N}{N-1}}\big)'(t)\,dt \to \int_{m'}^{m''}\big(I_{\mathscr H}^{\frac{N}{N-1}}\big)'(t)\,dt\quad\text{as $\eta\to0$\,.}
$$
Hence, \eqref{ugual6.1}  follows.

Returning to the proof of the flatness of $\pa\Om_0\cap\cc$, observe that by a simple diagonal argument we can construct two sequences $m_n\to m_0$ and $\eta_n\to0$  such that $\Om_{\eta_n,m_n}$ is a $\Lambda_n$-minimizer for some $\Lambda_n>0$ (possibly going to $+\infty$) and
$$
I_{\eta_n}(m_n)\to I_0(m_0),\quad\big(I_{\eta_n}^{\frac{N}{N-1}}\big)'(m_n)\to  \big(I_{\mathscr H}^{\frac{N}{N-1}}\big)'(m_0)=N\Big(N\frac{\omega_N}{2}\Big)^{\frac{1}{N-1}}\,.
$$
 This is possible thanks to Step 2, Step 6 and \eqref{ugual6.1}. Given $\e>0$, let $\delta>0$ be as in Theorem~\ref{th:cgr-main1} with $\theta_0=\pi/2$. Recall that $\delta$ depends only on $\e$ and on diam$(\Om_0)$. Recall also that $\Sigma_{\eta_n,m_n}$ is of class $C^{1,1}$ in a neighborhood of $\Sigma^+_{\eta_n,m_n}$, thanks to Step 7. Then from the above convergence, arguing as in the proof of \eqref{412} with $\widetilde\Sigma_\eta$ replaced by $\Sigma^+_{\eta_n,m_n}$, and recalling that the weak Young's inequality \eqref{wyl1} holds for $\Sigma_{\eta_n,m_n}$, we have that for $n$ large
\[
\begin{split}
\frac{N\omega_N}{2}\leq\mathcal K^+(\Sigma_{\eta_n,m_n})&\leq(N-1)^{1-N}\int_{\Sigma^+_{\eta_n,m_n}}H^{N-1}_{\Sigma_{\eta_n,m_n}}\,d\H^{N-1}\\
&\leq(N-1)^{1-N}P(\Om_{\eta_n,m_n};\R^N\setminus\cc_{\eta_n}) H_{\Sigma_{\eta_n,m_n}^*}^{N-1}\\
&=\Big[\frac{1}{N} \big(I_{\eta_n}^{\frac{N}{N-1}}\big)'(m_n)\Big]^{N-1}<\frac{N\omega_N}{2}+\delta\,.
\end{split}
\]
Note that in the third inequality above we have used \eqref{ugual6.2}.
Thus from Theorem~\ref{th:cgr-main1} we get that for $n$ sufficiently large $\wid(\Sigma_{\eta_n,m_n}\cap\cc_{\eta_n}):=\e_n\to0$ and more precisely that there exists $x_n\in\pa\cc_{\eta_n}$ such that 
\beq\label{ugual7.5}
\Sigma_{\eta_n,m_n}\cap\cc_{\eta_n}\subset\{x: -\e_n\leq (x-x_n)\cdot\nu_{\cc_{\eta_n}}(x_n)\leq0\}\,.
\eeq
Observe that, up to a not relabelled subsequence, 
\beq\label{ugual7.51}
x_n\to\overline{x}\in\cc,\qquad \nu_{\cc_{\eta_n}}(x_n)\to\overline\nu\in N_{\overline x}(\cc)\,.
\eeq
Denote by $\Pi$ the support hyperplane passing through $\overline x$ and orthogonal to $\overline\nu$ and by $\Pi^{\pm}$ the half spaces  $\{x:\, (x-\overline x)\cdot\overline\nu\gtrless0\}$. We claim that $\pa\Om_0\cap\cc\subset\Pi$ up to a set of $\H^{N-1}$-measure zero.

To prove the claim we first show that, passing possibly to a further subsequence,  
\beq\label{ugual7.55}
\pa\Om_{\eta_n,m_n}\cap\cc_{\eta_n}\to K\quad \text{ for some $K\subset\pa\Om_0\cap\cc$ s.t. $\H^{N-1}(\pa\Om_0\cap\cc \setminus K)=0$}\,.
\eeq
To this aim observe first that since $\cc_{\eta_n}\cap\overline{B_R(0)}$ is a sequence of convex sets converging to the convex set $\cc\cap\overline{B_R(0)}$ in the sense of Kuratowski then $P(\cc_{\eta_n}\cap\overline{B_R(0)})\to P(\cc\cap\overline{B_R(0)})$. This in turn yields that 
$\H^{N-1}\res\pa(\cc_{\eta_n}\cap\overline{B_R(0)})\stackrel{*}{\wto}\H^{N-1}\res\pa(\cc\cap\overline{B_R(0)})$
and in particular that
\beq\label{ugual17}
\H^{N-1}\res\pa\cc_{\eta_n}\stackrel{*}{\wto}\H^{N-1}\res\pa\cc\qquad\text{in $B_R(0)$\,.}
\eeq
%
 We claim that
\beq\label{ugual10}
\limsup_{n}\H^{N-1}(\pa\Om_{\eta_n,m_n}\cap\cc_{\eta_n})\leq\H^{N-1}(K)\,.
\eeq
To this aim set $K_\sigma=K+\overline{B_\sigma(0)}\subset B_R(0)$ for $\sigma>0$ sufficiently small. Then for $n$ sufficiently large $\pa\Om_{\eta_n,m_n}\cap\cc_{\eta_n}\subset K_\sigma\cap\pa\cc_{\eta_n}$, hence
$$
\H^{N-1}(\pa\Om_{\eta_n,m_n}\cap\cc_{\eta_n})\leq\H^{N-1}(K_\sigma\cap\pa\cc_{\eta_n})\,.
$$
From this inequality  we then have
$$
\limsup_{n}\H^{N-1}(\pa\Om_{\eta_n,m_n}\cap\cc_{\eta_n})\leq\limsup_{n}\H^{N-1}(K_\sigma\cap\pa\cc_{\eta_n})
\leq \H^{N-1}(K_\sigma\cap\pa\cc)\,,
$$
where in the last inequality we have used \eqref{ugual17}. Then \eqref{ugual10} follows letting $\sigma\to0$.
On the other hand, $\Om_{\eta_n,m_n}\to\Om_0$ in $L^1$ and by the lower semicontinuity of perimeter and \eqref{ugual10}
\begin{align*}
P(\Om_0)=I_0(m_0)+\H^{N-1}(\pa^*\Om_0\cap\cc)&\leq\liminf_{n}P(\Om_{\eta_n,m_n})\\
&=\liminf_{n}\big[I_{\eta_n}(m_n)+\H^{N-1}(\pa\Om_{\eta_n,m_n}\cap\cc_{\eta_n})\big]\\
&\leq I_0(m_0)+\H^{N-1}(K)\,.
\end{align*}
Recall that by the volume estimate Proposition~\ref{densityestimates}-(ii)  $\pa^*\Om_0\cap\cc$ coincides $\H^{N-1}$-a.e. with $\pa\Om_0\cap\cc$. Thus the above inequality implies that $K$ coincides $\H^{N-1}$-a.e. with  $\pa\Om_0\cap\cc$. Hence, \eqref{ugual7.55} follows.

We finally claim that for $n$ large
\beq\label{ugual8}
\pa\Om_{\eta_n,m_n}\cap\cc_{\eta_n}\subset\{x: -\e\leq (x-x_n)\cdot\nu_{\cc_{\eta_n}}(x_n)\leq0\}\,.
\eeq
To prove this we argue by contradiction assuming that for infinitely many $n$ there exists $y_n\in\pa\Om_{\eta_n,m_n}\cap\cc_{\eta_n}$ such that $(y_n-x_n)\cdot\nu_{\cc_{\eta_n}}(x_n)<-\e_n$. Observe that, if this is the case for all such $n$, 
\beq\label{ugual8.5}
F_n:=\pa\cc_{\eta_n}\cap\{x:  (x-x_n)\cdot\nu_{\cc_{\eta_n}}(x_n)<-\e_n\}\subset\pa\Om_{\eta_n,m_n}\cap\cc_{\eta_n}\,.
\eeq
Indeed if not there exists $z_n\in F_n\setminus\pa\Om_{\eta_n,m_n}$ and in turn a continuous path $\gamma\subset F_n$ connecting $z_n$ to $y_n$ (recall that $\cc_{\eta_n}$ is bounded). But then this arc must contain a point in $\pa_{\cc_{\eta_n}}(\pa\Om_{\eta_n,m_n}\cap\cc_{\eta_n})\subset \Sigma_{\eta_n,m_n}\cap\cc_{\eta_n}$, which contradicts \eqref{ugual7.5}. Therefore, from \eqref{ugual8.5}, \eqref{ugual7.51}  and \eqref{ugual7.55} we have that 
$$
\pa\cc\cap\{x:\, (x-\overline x)\cdot\overline\nu<0\}=\pa\cc\cap\Pi^-\subset \pa\Om_0\cap\cc\,.
$$
Then, let $\bar t:=\min\{t\leq0:\, \Pi+t\overline\nu\cap\cc\not=\emptyset\}$ and set for $t\in(\bar t,0)$, $\cc^t:=\cc\cap(\Pi^++t\overline\nu)$. Note that, from the above inclusion,  $P(\Om_0\cup(\cc\setminus\cc^t);\R^N\setminus\cc^t)=P(\Om_0;\R^N\setminus\cc)=I_{\mathcal H}(m_0)$, but this contradicts \eqref{isoperim1} since $\Om_0\cup(\cc\setminus\cc^t)>m_0$. Hence \eqref{ugual8} holds for $n$ large enough.

Finally, from \eqref{ugual8} and \eqref{ugual7.55} we have that  $\pa\Om_0\cap\cc\subset\Pi$ up to a set of vanishing $\H^{N-1}$ measure.

\par\noindent
{\bf Step 9.} ({\it Conclusion}). In this final step we show that $\Om_0$ is a half ball.

To this aim we fix $m\in(0,m_0)$ and a sequence $\eta_n\to0$ such that
\beq\label{ugual12}
I_{\eta_n}(m)\to I_0(m)=I_{\mathscr H}(m),\quad\big(I_{\eta_n}^{\frac{N}{N-1}}\big)'(m)\to  \big(I_{\mathscr H}^{\frac{N}{N-1}}\big)'(m)=N\Big(N\frac{\omega_N}{2}\Big)^{\frac{1}{N-1}}\,.
\eeq
Owing to Steps 6-8 we can find such a sequence for a.e. $m\in(0,m_0)$.
Thanks to Step 2, we may assume that there exists $\Lambda>0$ such that $\Om_{\eta_n,m}$ is a $\Lambda$-minimizer for all $n$. By Theorem~\ref{rm:regola}-(ii) this implies in particular that $|H_{\Sigma_{\eta_n,m}}|\leq\Lambda$ $\H^{N-1}$-a.e. on $\pa^*\Om_{\eta_n,m}\setminus\cc_{\eta_n}$.
Arguing as in the previous step, see also the proof of \eqref{412}, we have then
\beq\label{ugual13}
\begin{split}
\frac{N\omega_N}{2}\leq\mathcal K^+(\Sigma_{\eta_n,m})
&=\int_{\Sigma^+_{\eta_n,m}}K_{\Sigma_{\eta_n,m}}\,d\H^{N-1}
\leq(N-1)^{1-N}\int_{\Sigma^+_{\eta_n,m}}H^{N-1}_{\Sigma_{\eta_n,m}}\,d\H^{N-1}\\
&\leq(N-1)^{1-N}P(\Om_{\eta_n,m};\R^N\setminus\cc_{\eta_n}) H_{\Sigma_{\eta_n,m}^*}^{N-1}\\
&=\Big[\frac{1}{N} \big(I_{\eta_n}^{\frac{N}{N-1}}\big)'(m)\Big]^{N-1}\to\frac{N\omega_N}{2}\,,
\end{split}
\eeq
where we recall $K_{\Sigma_{\eta_n,m}}$ is the Gaussian curvature of $\Sigma_{\eta_n,m}$. We start by observing that, since $H_{\Sigma_{\eta_n,m}}(x)\leq H_{\Sigma_{\eta_n,m}^*}$ for $\H^{N-1}$-a.e. $x\in\Sigma^+_{\eta_n,m}$, from the third inequality in \eqref{ugual13} we have in particular that 
\beq\label{ugual14}
\lim_{n}\H^{N-1}(\Sigma^+_{\eta_n,m})=\lim_{n}P(\Om_{\eta_n,m};\R^N\setminus\cc_{\eta_n})=\H^{N-1}(\pa^*\Om_{0,m}\setminus\cc)\,.
\eeq
Note that indeed $H_{\Sigma_{\eta_n,m}}$ may only take the constant values $H_{\pa\Om_0}$ or $H_{\Sigma_{\eta_n,m}^*}$. Then, again from \eqref{ugual13}, it follows that
\beq\label{ugual15}
\text{either $\quad H_{\pa\Om_0}=H_{\Sigma_{\eta_n,m}^*}\quad$ or}\quad \H^{N-1}\big((\pa\Om_{\eta_n,m}\cap\pa\Om_{0})\setminus\cc_{\eta_n}\big)\to0\,.
\eeq
Fix now $x\in\pa^*\Om_{0,m}$. Since $\Om_{\eta_n,m}\to\Om_{0,m}$ in $L^1$ and $P(\Om_{\eta_n,m};\R^N\setminus\cc_{\eta_n})\to P(\Om_{\eta_n,m};\R^N\setminus\cc)$ thanks to the first condition in \eqref{ugual12}, we have that $\H^{N-1}\res\pa^*\Om_{\eta_n,m}\stackrel{*}{\wto}\H^{N-1}\res\pa^*\Om_{0,m}$ in $\R^N\setminus\cc$.
In turn, by Theorem~\ref{th:cicaleo} it follows that, up to rotations and translations, there exist a $(N-1)$-dimensional ball $B'\subset\R^{N-1}$, functions $\varphi_n,\varphi\in W^{2,p}(B')$, and $r>0$ such that $x\in B'\times(-r,r)$ and
\begin{align*}
&\pa\Om_{\eta_n,m}\cap(B'\times(-r,r))=\{(x',\varphi_n(x')):\,x'\in B'\},\\
&\pa\Om_{0,m}\cap(B'\times(-r,r))=\{(x',\varphi(x')):\,x'\in B'\},\\
&\varphi_n\wto\varphi\quad\text{in $W^{2,p}(B')$ for all $p\geq1$},\\
&H_{\Sigma_{\eta_n,m}}(x',\varphi_n(x'))\wto H_{\Sigma_{0,m}}(x',\varphi(x'))\quad\text{in $L^p(B')$ for all $p\geq1$\,,}
\end{align*}
Recalling \eqref{ugual15}   the fourth condition above implies that 
$$H_{\Sigma_{\eta_n,m}}(x',\varphi_n(x'))\to H_{\Sigma_{0,m}}(x',\varphi(x'))\equiv H^{N-1}_{\Sigma^*_{0,m}}$$ strongly in $L^p(B')$ for all $p\geq1$. In turn, see for instance \cite[Lemma 7.2]{AFM}, this implies
\beq\label{ugual16}
\varphi_n\to\varphi\quad\text{strongly in $W^{2,p}(B')$ for all $p\geq1$\,.}
\eeq
Note also that, since from \eqref{ugual14} $\H^{N-1}(\Sigma_{\eta_n,m}\setminus\Sigma^+_{\eta_n,m})\to0$, we have that for every $y\in (B'\times(-r,r))\cap\Sigma_{0,m}$ there exists a sequence $y_n\in (B'\times(-r,r))\cap\Sigma^+_{\eta_n,m}$ such that $y_n\to y$. Therefore, using the $L^1$ convergence of $\Om_{\eta_n,m}$ to $\Om_{0,m}$ we conclude that the tangent hyperplane to $\pa\Om_{0,m}$ at $y$ is also a support hyperplane. Thus we have shown that all principal curvatures at any point in $(B'\times(-r,r))\cap\pa\Sigma_{0,m}$ are nonnegative. Thus, from the second inequality in \eqref{ugual13}, recalling \eqref{ugual14} and \eqref{ugual16} we may conclude that
$$
K_{\Sigma_{0,m}}=(N-1)^{1-N}H^{N-1}_{\Sigma_{0,m}}=(N-1)^{1-N}H^{N-1}_{\Sigma^*_{0,m}}\quad\text{on $(B'\times(-r,r))\cap\Sigma_{0,m}$}\,.
$$
The equality above implies that $\Sigma_{0,m}\cap(B'\times(-r,r)))$ is umbilical. Hence $\pa^*\Om_{0,m}\setminus\cc$ is umbilical, thus each connected component of $\pa^*\Om_{0,m}\setminus\cc$ lies on  a sphere of radius $R_m=(N-1)/H_{\Sigma^*_{0,m}}$.  Consider the unique unbounded connected component of $U:=\R^N\setminus\overline{\Om_{0,m}}$. Then, recalling Step 8, $\pa U\setminus\cc$ is contained in a sphere of radius $R_m$ intersecting $\cc$ on $\Pi$. Thus $\pa U\setminus\cc$ is a spherical cap and $\Om_{0,m}$ is contained in the region enclosed by  $\pa U\setminus\cc$ and $\Pi$. In particular $\Om_{0,m}$ is contained in the half space $\Pi^+$ determined by $\Pi$ not containing $\cc$. Since $P(\Om_{0,m};\Pi^+)=P(\Om_{0,m};\R^N\setminus\cc)=N\big(\frac{\omega_N}{2}\big)^{\frac{1}{N}}m^{\frac{N-1}{N}}$, by Theorem~19.21 in \cite{maggi-book} for a.e. $m$ we conclude that for such $m$ $\Om_{0,m}$ is a half ball. Since the argument above can be carried out for a.e. $m\in(0,m_0)$, in particular there exists a sequence $m_n\to m_0$ such that $\Om_{0,m_n}$ is a half ball. Hence  $\Om_0$ is a half ball.

 \end{proof}

\section{Appendix A: some auxiliary results}\label{sc:5}
In this section we collect some auxiliary results needed in the proof of Theorem~\ref{th:ugual}.
\subsection{Density estimates}
Density estimates for $(\Lambda,r_0)$-minimizers are well known. However for the sake of completeness we give  the proof of the  proposition below showing that such density estimates are independent of the convex obstacle. 

\begin{lemma}\label{lm:perconv}
Let $\cc$ be a closed convex set with nonempty interior and $F\subset\R^N\setminus\cc$ a bounded set of finite perimeter. Then
$$
P(F;\pa\cc)\leq P(F;\R^N\setminus\cc)
$$
\end{lemma}
\begin{proof}
Let $B$ a ball such that $F\subset\!\subset B$ and let $H_i$ be a sequence of closed half spaces such that $\cc=\displaystyle\bigcap_{i=1}^{\infty}H_i$. Since $\cc=(\cc\cup F)\cap\displaystyle\bigcap_{i=1}^{\infty}H_i$ we have
$$
P(\cc;B)\leq\liminf_nP\Big((\cc\cup F)\cap\bigcap_{i=1}^{n}H_i;B\Big)\leq P(\cc\cup F;B)\,,
$$
where the last inequality follows by applying repeatedly the inequality $P(G\cap H_i;B)\leq P(G;B)$ where $G$ is a set of finite perimeter.
Since $P(\cc\cup F;B)=\H^{N-1}(\pa\cc\cap F^{(0)}\cap B)+\H^{N-1}(\pa^*F\setminus\cc)$, the conclusion follows observing that $P(\cc;B)=\H^{N-1}(\pa\cc\cap F^{(0)}\cap B)+\H^{N-1}(\pa\cc\cap\pa^*F)$. \end{proof}

\begin{proposition}\label{densityestimates}
Let $\cc$ be a closed convex set with nonempty interior and
let  \(E\subset\R^N\setminus\cc\) be a restricted $(\Lambda,r_0)$-minimizer of the relative perimeter  $P(\cdot;\R^N\setminus\cc)$ according to Definition~\ref{def:lambdamin}. Then there are   positive constants \(c_1=c_1(N)\) and \(C_1=C_1(N)\) independent of  $\cc$ such that for all $r\in(0,\min\{r_0,N/(4\Lambda)\})$ we have:
\begin{itemize}
\item[(i)] for all $x\in \R^N\setminus\text{int}(\cc)$ 
$$
P(E; B_r(x))\leq C_1 r^{N-1}\;,
$$
\item[(ii)] for all \(x\in \partial^*E\)  
$$
|E\cap B_r(x)|\geq c_1 r^N\,.
$$
\end{itemize}
Moreover $E$ is equivalent to an open set $\Om$ such that $\pa\Om=\pa^e\Om$, hence $\H^{N-1}(\pa\Om\setminus\pa^*\Om)=0$, and (ii) holds at any point $x\in\pa\Om$.
\end{proposition}

\begin{proof}  Given $x\in\R^N\setminus\text{int}(\cc)$ and $r<\min\{r_0,N/(4\Lambda)\}$, we set $m(r):=|E\cap B_r(x)|$. Recall that for a.e. such $r$ we have  
$m'(r)=\H^{N-1}(E^{(1)}\cap \partial B_r(x))$ and $\H^{N-1}(\partial^*E\cap \partial B_r(x))=0$. For any such $r$ we set $F:=E\setminus B_r(x)$. Then, using Definition~\ref{def:lambdamin}, we have 
\begin{equation}\label{e:de1}
P(E; B_r(x)\setminus\cc)\leq \H^{N-1}(\partial B_r(x)\cap E^{(1)})+\Lambda|E\cap B_r(x)|
\leq   C_1 r^{N-1} 
\end{equation}
for a suitable  constant $C_1$. In turn 
$$
P(E; B_r(x))\leq P(E; B_r(x)\setminus\cc)+\H^{N-1}(\pa(\cc\cap B_r(x)))\leq C_1 r^{N-1}+\H^{N-1}(\pa B_r(x))\,,
$$
where in the last inequality we estimated the perimeter $\cc\cap B_r(x)$ with the perimeter of the larger convex set $B_r(x)$. Thus (i) follows by taking $C_1$ larger.

Observe now that  by Lemma~\ref{lm:perconv}
$$
P(E\cap B_r(x); \partial\cc)\leq P(E\cap B_r(x); \R^N\setminus\cc)\,.
$$
Thus, using also \eqref{e:de1}, we have
\begin{align*}
P(E\cap B_r(x))&=P(E\cap  B_r(x);\R^N\setminus\cc)+P(E\cap  B_r(x); \partial \cc)\\
&\leq 2 P(E\cap  B_r(x); \R^N\setminus\cc)= 2 P(E;  B_r(x)\setminus\cc)+2m'(r)\\
&\leq 4 m'(r)+2\Lambda m(r)\,.
\end{align*}
In turn, using the isoperimetric inequality and the fact that $2\Lambda r<N/2$ we get
\begin{align*}
N\omega_N^{\frac1N}m(r)^{\frac{N-1}{N}}& \leq P(E\cap B_r(x))\leq 4m'(r)+2\Lambda m(r)\\
&\leq 4m'(r)+2\Lambda r\omega_N^{\frac1N}m(r)^{\frac{N-1}N}\leq
4m'(r)+\frac{N}{2}\omega_N^{\frac1N}m(r)^{\frac{N-1}N}\,.
\end{align*}
Then  from the previous inequality we get
 $$
\frac{N}{2}\omega_N^{\frac1N}m(r)^{\frac{N-1}{N}}\leq 4m'(r)\,.
 $$
 Observe now that if in addition $x\in \partial^* E$, then $m(r)>0$ for all $r$ as above. Thus, we may divide the previous inequality by $m(r)^{\frac{N-1}{N}}$, and integrate the resulting differential inequality thus getting 
 $$
 |E\cap B_r(x)|\geq c_1r^N\,,
 $$
 for a suitable positive constant $c_1$ depending only on $N$. 
 
We show  that $\overline{\pa^*E}\subset\pa^eE$. To this aim note that (ii) holds for every $x\in\overline{\pa^*E}$. Thus, if $x\in\R^N\setminus\cc$, since both $E$ and $\R^N\setminus E$ are $\Lambda$-minimizers in a neighborhood of $x$ we have that $|E\setminus B_r(x)|\geq c_1r^N$ for $r$ small. Thus $x\not\in(E^{(0)}\cup E^{(1)})$, that is $x\in\pa^eE$. If $x\in\pa\cc\cap\overline{\pa^*E}$ then there exists a constant $c_2>0$, depending on $x$ such that for $r$ small $|\cc\cap B_r(x)|\geq c_2r^N$. This estimate, together with (ii) again implies that $x\in\pa^eE$. Hence $\H^{N-1}(\overline{\pa^*E}\setminus\pa^*E)\leq\H^{N-1}(\pa^eE\setminus\pa^*E)=0$, where the last equality follows from Theorem~16.2 in \cite{maggi-book}.

Set now $\Om=E^{(1)}\setminus\pa E^{(1)}$. Recalling that $\pa E^{(1)}=\overline{\pa^*E}$, see \eqref{Euno}, we have that $\Omega$ is an open set equivalent to $E$  such $\pa\Om=\pa E^{(1)}$. Hence the conclusion follows.
 \end{proof}

\subsection{A maximum principle} Next result is essentially the strong maximum principle proved in \cite[Lemma 2.13]{dephilippis-maggi-arma}. However, we have to apply it in a slightly different situation and therefore we indicate the  changes needed in the proof.
\begin{lemma}\label{th:maxprin}
Let $E\subset\{x_1>0\}$ be a set of locally finite perimeter such that 
\beq\label{maxprin0}
\H^{N-1}((\pa E\setminus\pa^*E)\setminus\{x_N=0\})=0
\eeq
 satisfying the following minimality property:
for every $r>0$ and  every open set with Lipschitz boundary $V\subset\!\subset Q_r(0)$ such that $\H^{N-1}(\pa E\cap\pa V)=0$  we have
\beq\label{maxprin1}
P(E; Q_r(0))\leq P(E\setminus V; Q_r(0))\,.
\eeq
Assume also that $\H^{N-1}(\pa E\cap\{x_1=0\})=0$. Then  $0\not\in\pa E$\footnote{Here as usual we assume that $\pa E=\overline{\pa^*E}$.}.
\end{lemma}
The proof of lemma above is in turn based on the following variant of \cite[Lemma 2.12]{dephilippis-maggi-arma}. To this aim, given $r>0$ we set $C_r:=(0,r)\times D_r$, where $D_r:=\{x'\in\R^{N-1}:\, |x'|<r\}$.
\begin{lemma}\label{lm:DPCM}
Let $E$ be as in Lemma~\ref{th:maxprin}, let $\bar r>0$ and let $u_0\in C^2(D_{\bar r})\cap\text{\rm Lip}(D_{\bar r})$ with $0<u_0<{\bar r}$ on $\overline D_{\bar r}$. Assume also that
$$
E^{(1)}\cap[(0,{\bar r})\times\pa D_{\bar r}]\subset\{(x_1,x')\in(0,{\bar r})\times\pa D_{\bar r}:\,x_1\geq u_0(x')\}\,,
$$
$$
\Div\bigg(\frac{\nabla u_0}{\sqrt{1+|\nabla u_0|^2}}\bigg)=0\qquad\text{in $D_{\bar r}$}
$$
and
\beq\label{DPCM1}
\H^{N-1}\big(\pa E\cap\pa\{(x_1,x')\in C_{\bar r}:\,x_1< u_0(x')\}\big)=0\,.
\eeq
Then, $$
E^{(1)}\cap C_{\bar r}\subset\{(x_1,x')\in C_{\bar r}:\,x_1\geq u_0(x')\}\,.
$$
\end{lemma}
\begin{proof}
The proof goes exactly as the one of Lemma 2.12 in \cite{dephilippis-maggi-arma} as it is based on the comparison with he competitor $F=E\setminus V$, where $V=\{(x_1,x')\in C_{\bar r}:\,x_1< u_0(x')\}$. Observe that assumption \eqref{DPCM1} guarantees that such a competitor satisfies  $\H^{N-1}(\pa E\cap \pa V)=0$, which is required  in order  \eqref{maxprin1} to hold.
\end{proof}
\begin{proof}[Proof of Lemma~\ref{th:maxprin}]
For reader's convenience we reproduce the proof of Lemma 2.13 in \cite{dephilippis-maggi-arma} with the small changes needed in our case.

We choose ${\bar r}>0$ so that $\H^{N-1}(\pa E\cap\pa C_{\bar r})=0$  and $\H^{N-2}(\pa E\cap\pa D_{\bar r})=0$, where with a slight abuse of notation $\pa D_{\bar r}$ stands for the relative boundary of $D_{\bar r}$ in $\{x_1=0\}$.  Note that a.e. $r>0$ satisfies these conditions thanks to  \eqref{maxprin0} and to the assumption  $\H^{N-1}(\pa E\cap\{x_1=0\})=0$. Define now a function $w_E : \overline{D_{\bar r}}\to[0,\infty]$ by setting
$$
w_E(x')=\inf\{x_1\in\R:\, (x_1,x')\in \overline{C_{\bar r}}\cap\pa E\}.
$$
Observe that $w_E$ is nonnegative and lower semicontinuous on $\overline{D_{\bar r}}$, with the property that
$$
E^{(1)}\cap C_{\bar r}\subset\{(x_1,x'):\,x'\in D_{\bar r},\,x_1\geq w_E(x')\}\,.
$$

Recalling that  $\H^{N-2}(\pa E\cap\pa D_{\bar r})=0$, we have that $w_E>0$ $\H^{N-2}$-a.e. on $\pa D_{\bar r}$.
Therefore there exists a family  $(\varphi_t)_{t\in(0,1)}\subset C^\infty(\pa D_{\bar r})$ such that 
$$
 0\leq\varphi_{t_1}\leq\varphi_{t_2}\leq \min\Big\{w_E,\frac{\bar r}{2}\Big\}\quad \varphi_{t_1}\not\equiv\varphi_{t_2}\quad\text{for all $0<t_1<t_2<1$}\,.
$$
By Lemma~2.11 in \cite{dephilippis-maggi-arma} for every $t\in(0,1)$ there exists $u_t\in C^{\infty}(D_{\bar r})\cap{\rm Lip}(D_{\bar r})$ such that
$$
\begin{cases}
\Div\bigg(\displaystyle\frac{\nabla u_t}{\sqrt{1+|\nabla u_t|^2}}\bigg)=0& \text{in $D_{\bar r}$,} \\
u_t=\varphi_t & \text{on $\pa D_{\bar r}$\,.}
\end{cases}
$$
Note that by the Strong Maximum Principle Theorem~\ref{th:SMP} we have that  $0<u_{t_1}<u_{t_2}<{\bar r}/2$ in $D_{\bar r}$ for every $0<t_1<t_2<1$. Therefore  the graphs $\Gamma_t$ of $u_t$ are mutually disjoint in $C_{\bar r}$ and so $\H^{N-1}(\Gamma_t\cap\pa E)=0$ for all but countably many $t\in(0,1)$. In particular there exists $\bar t$ such that \eqref{DPCM1} holds with $u_0$ replaced by $u_{\bar t}$. Therefore we may apply Lemma~\ref{lm:DPCM} to conclude that $E^{(1)}\cap C_{\bar r}\subset\{(x_1,x')\in C_{\bar r}:\,x_1\geq u_{\bar t}(x')\}$ so that in particular $w_E(0)\geq u_{\bar t}(0)>0$, hence $0\not\in\pa E$.
\end{proof}

\subsection{A regularity result}
The following proposition is a slight variant of a result contained in \cite{SZi}.
\begin{proposition}\label{strezi}
Let $\Omega\subset\R^N$ be a bounded open set  and let $x_0\in\pa\Om$ be such that $\pa\Om$ is of class $C^2$ in a neighborhood $U$ of $x_0$. Let $E\subset\Om$ satisfy
\beq\label{strezi1}
P(E)\leq P(F)\quad\text{for all $F\subset\Omega$, $|F|=|E|$,  s.t. $E\Delta F\subset\!\subset U$\,.}
\eeq
If there exists a support hyperplane $\Pi$ to $E$ at $x_0$ such that $\pa E\cap\Pi=\{x_0\}$, then $\pa E$ is of class $C^{1,1}$ in a neighborhood $V$ of $x_0$. Moreover if $\pa^*E\cap\Om\cap V\not=\emptyset$, then for $\H^{N-1}$-a.e. $x\in\pa E\cap\pa\Om\cap V$
\beq\label{strezi2}
H_{\pa\Om}(x)\leq H\,,
\eeq
where $H$ denotes the constant curvature of $\pa^*E\cap\Om\cap V$.
\end{proposition}
\begin{proof}
%
Observe that by a standard argument \eqref{strezi1}, together with the assumption that $\pa\Om$ of class $C^2$, implies that $E$ is a $(\Lambda,r_0)$-minimizer in a possibly smaller naeighborhood $U'$ of $x_0$. Hence, since there exists a support hyperplane to $\pa E$ at $x_0$,   by Theorem~\ref{rm:regola} $\pa E$ is of class $C^{1,\alpha}$ in a neighborhood of $x_0$. Moreover,  up to a change of coordinate system, we may assume that the support hyperplane to $E$ at $x_0$ is the horizontal hyperplane $\{x_N=0\}$ and $E\subset\{x_N>0\}$. Since $\{x_N=0\}\cap\pa E=\{x_0\}$, there exists $\e>0$ sufficiently small such that $E\cap\{x_N=\e\}$ is an $(N-1)$-dimensional relatively open set, denoted by $\omega$, and there exist  $\beta\in C^2(\omega)$, $u\in C^{1,\alpha}(\omega)$ whose graphs coincide with $\pa\Om\cap(\omega\times(-r,r))$ and $\pa E\cap(\omega\times(-r,r))$ respectively, for some $r>0$, with $u=0$ on $\pa\omega$ and $\beta\leq u\leq0$. The $C^{1,1}$ regularity of $\pa E$ then follows arguing exactly as in the proof at page 658 of \cite{SZi}. Finally, inequality \eqref{strezi2} is also a byproduct of the same proof, see $(3.5)$ in  \cite{SZi}.
\end{proof}

\section{Appendix B: some steps of the proof of Theorem~\ref{th:ugual}}\label{sc:6}

\subsection{Proof of the claim of Step 1.}

We argue by contradiction assuming that there exist a sequence $\Lambda_h\to+\infty$, $\eta_h\in[0,\tilde\eta]$, $\eta_h\to\eta_0$,  $m_h\in[m',m'']$ converging to some $m$, and a sequence $E_h\subset\Om_0\setminus\cc_{\eta_h}$ such that each $E_h$ is a minimizer of \eqref{ugual1} with $\Lambda', m$ and $\eta$ replaced by $\Lambda_h, m_h$ and $\eta_h$ respectively, and $|E_h|\not=m_h$. {Since $P(E_h;\R^N\setminus\cc_{\eta_h})\leq P(\Om_{\eta_h,m_h};\R^N\setminus\cc_{\eta_h})$, from \eqref{supsup} we have that the perimeters of  $E_h$ are equibounded perimeters. Therefore,} without loss of generality we may assume that $E_h$ converges in $L^1$ to some set $F\subset\Om_0\setminus\cc_{\eta_0}$ such that $|F|=m$. We assume also that $|E_h|<m_h$ for all $h$, the other case being analogous. Note also that, since $\Lambda_h\to+\infty$ we have $m_h-|E_h|\to0$.

Observe now that  \eqref{lunga} implies that  there exists a point $x\in\partial^*F\cap(\Om_0\setminus\cc_{\eta_0})$. 
Arguing as in Step~1 of Theorem~1.1 in  \cite{EF}, given $\e>0$ sufficiently small, we can find nearby $x$ a point $x'$ and  $r>0$ such that $B_r(x')\subset\Om_0\setminus\cc_{\eta_0}$ and 
  $$
 | F\cap B_{r/2}(x')|<\e r^N\,, \quad |  F \cap B_r(x')|>\frac{\omega_N}{2^{N+2}}r^N\,. 
 $$
 Therefore, for $h$ sufficiently large, we also have
 $$
 | E_h\cap B_{r/2}(x')|<\e r^N\,, \quad | E_h\cap B_r(x')|>\frac{\omega_N}{2^{N+2}}r^N\,.
  $$
 We can now continue as in the proof of \cite[Theorem~1]{EF}. We recall the main construction for the reader's convenience.   For a sequence $0<\sigma_h<1/2^N$ to be chosen, we introduce the following bilipschitz maps:
  $$
  \Phi_x(x):=
 \begin{cases}
 x'+(1-\sigma_h(2^N-1))(x-x') & \text{if $|x-x'|\leq \frac r2$,} \vspace{5pt}\\
 x+\sigma_h\Bigl(1-\displaystyle\frac{r^N}{|x-x'|^N}\Bigr)(x-x') & \text{$\frac r2\leq |x-x'|<r$,}\vspace{5pt}\\
 x & \text{$|x-x'|\geq r$.}
 \end{cases}
 $$
 Setting $\widetilde E_h:=\Phi_h(E_h)$, arguing as for the proof of \cite[formula~(14)]{EF}, we have 
 \beq\label{EFper}
 \H^{N-1}\big(\pa^*E_h\setminus\cc_{\eta_0}\big)- \H^{N-1}(\pa^*\widetilde E_h\setminus\cc_{\eta_0})\geq -2^NN\sigma_h \H^{N-1}(\pa^*E_h\setminus\cc_{\eta_0})\,.
 \eeq
 Moreover, arguing exactly as in Step 4  of the proof of  \cite[Theorem~1]{EF} we have
$$
 |\widetilde E_h|-|E_h|\geq \sigma_hr^N(c-\e C)
$$
 for  suitable universal constants $c, C>0$.
If we fix $\e$ so that the negative term  in the brackets does not exceed half  the  positive one, then we have 
  \beq\label{EFvol}
   |\widetilde E_h|-|E_h|\geq \frac{c}2\sigma_hr^N\,.
  \eeq
  In particular from this inequality it is clear that we can choose $\sigma_h$ so that $|\widetilde E_h|=m_h$; this implies  $\sigma_h\to 0$. With this choice of $\sigma_h$, it follows from \eqref{EFper} and \eqref{EFvol} that 
  \begin{align*}
  P(\widetilde E_h;\R^N\setminus\cc_{\eta_h})+\Lambda_h||\widetilde E_h|-m_h|&\leq 
  P( E_h;\R^N\setminus\cc_{\eta_h})+\Lambda_h|| E_h|-m_h| \\
  & \qquad+2^NN\sigma_h \H^{N-1}(\pa^*E_h\setminus\cc_{\eta_h})-\Lambda_h \frac{c}2\sigma_hr^N\\
  & <P( E_h;\R^N\setminus\cc_{\eta_h})+\Lambda_h|| E_h|-m_h| 
\end{align*}
  for $h$ large, thus contradicting the minimality of $E_h$. 

\subsection{Proof of the claim of Step 3.} We start by showing that the functions $I_\eta$ are strictly increasing in $[0,m_\eta]$ for all $\eta\in[0,\bar\eta]$. To this end we fix $m\in(0,m_\eta]$ 
 and a point $x\in\pi_{\cc_\eta}(\Om_{\eta,m})$, where $\pi_{\cc_\eta}$ is the orthogonal projection on $\cc_\eta$. Let $\Pi$ be the tangent hyperplane to $\cc_\eta$ at $x$. Define $\Pi_t=\Pi+t\nu_{\cc_\eta}(x)$ for $t\in\R$ and set
$$
\bar t=\max\{t\geq0:\,\Pi_t\cap \overline{\Om_{\eta,m}}\not=\emptyset\}.
$$
Note that $\bar t>0$ and that  $\Pi_{\bar t}$ is a support hyperplane for $\Om_{\eta,m}$ with dist$(\Pi_{\bar t},\cc_\eta)=\bar t$. 
%
For all $t\in(0,\bar t)$ we denote by $\Om_{\eta,m,t}$ the intersection of $\Om_{\eta,m}$ with the half space with boundary $\Pi_t$ containing $\cc_\eta$. Then $I_{\eta}(|\Om_{\eta,m,t}|)\leq P(\Om_{\eta,m,t};\R^N\setminus\cc_\eta)<P(\Om_{\eta,m};\R^N\setminus\cc_\eta)=I_{\eta}(m)$. Since the function $t\to|\Om_{\eta,m,t}|$ is   increasing and continuous  in a left neighborhood of $\bar t$ and $|\Om_{\eta,m,t}|<|\Om_{\eta,m}|$ if $t<\bar t$,  it follows  that 
\beq\label{isoprof2}
\text{for every $m\in(0,m_\eta]$ there exists $\e>0$ s.t. $I_{\eta}(s)<I_{\eta}(m)$ for all $s\in(m-\e,m)$}\,.
\eeq
 Let $I=\{0<s<m:\,I_{\eta}(\sigma)\leq I_{\eta}(m)$  for all $\sigma\in[s,m)\}$. We claim that $I=(0,m)$. Indeed if $\bar m=\inf I>0$, then there exist $m_n\in I$, with $m_n\to\bar m^+$. Since the minimizers $\Om_{\eta, m_n}$ are equibounded sets with equibounded perimeters, see \eqref{supsup}, up to a subsequence we may assume that $\Om_{\eta, m_n}$ converge to a set $E\subset \Om_0\setminus\cc_\eta$ with $|E|=\bar m$. Then, by the lower semicontinuity of the perimeter we conclude that $I_{\eta}(\bar m)\leq P(E;\R^N\setminus\cc_\eta)\leq \liminf_{n}I_{\eta}(m_n)\leq I_\eta(m)$. In turn, \eqref{isoprof2} implies that there exists exists a left neighborhood $(\bar m-\e,\bar m)$ such that $I_{\eta}(s)<I_{\eta}(\bar m)\leq I_\eta(m)$ for all $s\in(\bar m-\e,\bar m)$ which is a contradiction to the fact that $\bar m=\inf I$. This contradiction proves that $I_{\eta}$ is increasing. The strict monotonicity now follows from \eqref{isoprof2}.
   
 Finally if, $m_1,m_2\in[m',m'']$, from \eqref{ugual1} we have for $\eta\in[0,\tilde\eta]$
 $$
 I_\eta{m_2}=P(\Om_{\eta,m_2};\R^N\setminus\cc_\eta)\leq P(\Om_{\eta,m_1};\R^N\setminus\cc_\eta)+\Lambda'|m_2-m_1|\,.
 $$
 This proves the $\Lambda'$-Lipschitz continuity of $I_\eta$.

\subsection{Proof of claim \eqref{calim4}}
Let us start by assuming also that
\beq\label{calim1}
\H^{N-1}(\pa E_h\cap\pa V\cap\mathscr H)=0 \qquad\text{for all $h\in\mathbb N$\,.}
\eeq
To this aim we fix $\delta>0$ and set $\mathscr H_\delta:=\{x\in\mathscr H:\, x_N>\delta\}$ and $(E)_\delta=E+B_\delta(0)$. Then we denote by $\Phi_h:\overline{Q_r(0)\cap\mathscr H}\to\overline{Q_r(0)\setminus\cc_{\eta,h}}$ a sequence of $C^1$ diffeomorphisms converging in $C^1$ to the identity map as $h\to0$ with the property that $\Phi_h(\pa\mathscr H\cap Q_r(0))=\pa\cc_{\eta,h}\cap Q_r(0)$ and $\Phi_h(x)=x$ if $x\in\mathscr H_\delta$. 
 Recalling the $\Lambda$-minimality property \eqref{lambdaetah},   we have using \eqref{calim1} and observing that $\Phi_h(V)\subset\!\subset Q_r(0)$ for $h$ sufficiently large
\begin{align*}
&P(E_h; Q_r(0)\setminus \cc_{\eta,h})\leq P(E_h\setminus\Phi_h(V); Q_r(0)\setminus \cc_{\eta,h})+\frac{\Lambda}{h}|\Phi_h(V)|\\
&\qquad\leq P(E_h; (Q_r(0)\setminus\cc_{\eta,h})\setminus\overline{\Phi_h(V)}) +
 P(\Phi_h(V); (Q_r(0)\setminus \cc_{\eta,h})\cap E_h)\\
 &\qquad\qquad+\H^{N-1}(\pa\Phi_h(V)\cap\pa E_h\cap\{x_N\leq\delta\}\cap (Q_r(0)\setminus\cc_{\eta,h}))+\frac{\Lambda}{h}|\Phi_h(V)|\,.
\end{align*}
Since
$$
P(E_h; Q_r(0)\setminus \cc_{\eta,h})\geq P(E_h; (Q_r(0)\setminus\cc_{\eta,h})\setminus\overline{\Phi_h(V)})+P(E_h; (Q_r(0)\setminus\cc_{\eta,h})\cap\Phi_h(V))\,,
$$
and using the fact that $\mathscr H_\delta\cap V=\mathscr H_\delta\cap\Phi_h(V)\subset (Q_r(0)\setminus\cc_{\eta,h})\cap\Phi_h(V)$,
 the inequality above yields
\begin{align}\label{calim5}
&P(E_h;\mathscr H_\delta\cap V)\leq P(\Phi_h(V); (Q_r(0)\setminus \cc_{\eta,h})\cap E_h)\nonumber \\
&\qquad\qquad+\H^{N-1}(\pa\Phi_h(V)\cap\{x_N\leq\delta\}\cap (Q_r(0)\setminus\cc_{\eta,h}))+\frac{\Lambda}{h}|\Phi_h(V)| \nonumber\\
&\qquad\leq P(V; Q_r(0)\cap\mathscr H_\delta\cap E_h)\nonumber\\
&\qquad\qquad+2\H^{N-1}(\pa\Phi_h(V)\cap\{x_N\leq\delta\}\cap (Q_r(0)\setminus\cc_{\eta,h}))+\frac{\Lambda}{h}|\Phi_h(V)| \\
&\qquad\leq P(V; Q_r(0)\cap\mathscr H_\delta\cap (E)_\delta)\nonumber\\
&\qquad+2(\text{Lip}(\Phi_h))^{N-1}P(V;\{0<x_N\leq\delta\})+\frac{\Lambda}{h}|\Phi_h(V)|\,,\nonumber
\end{align}
where  in the last inequality we used the fact that $\Phi_h^{-1}((Q_r(0)\setminus \cc_{\eta,h})\cap\{x_N\leq\delta\})= Q_r(0)\cap\{0<x_N\leq\delta\}$ and the fact that $\overline{E_h}$ converge in the Kuratowski sense to $\overline{E}$ in $\mathscr H_\delta$, see Remark~\ref{rm:compact}.
By the lower semicontinuity of the perimeter,  passing to the limit in  \eqref{calim5}
$$
P(E;\mathscr H_\delta\cap V)\leq P(V; Q_r(0)\cap\mathscr H_\delta\cap (E)_\delta)+2P(V; \{0<x_N\leq\delta\})\,.
$$
In turn, by letting $\delta\to0$ we have
\beq\label{calim6}
P(E;\mathscr H\cap V)\leq P(V; Q_r(0)\cap E)\,,
\eeq
 which is equivalent to \eqref{calim4} thanks to  first condition in \eqref{calim4}. 
 To remove \eqref{calim1} it is enough to consider a sequence of smooth sets $V_j\subset\!\subset Q_r(0)$, $V\subset\!\subset V_j$, satisfying the first condition in \eqref{calim4} and \eqref{calim1}, and such that $V_j\to V$ in $L^1$ and $P(V_j;Q_r(0))\to P(V;Q_r(0))$. The conclusion then follows by applying \eqref{calim4} with $V$ replaced by $V_j$ and passing to the limit  thanks to the first condition in \eqref{calim4}.

\end{document}